\documentclass[10pt]{jja}

\usepackage{amsmath,amstext, amsfonts, verbatim,remreset}
\usepackage{graphicx}
\usepackage{color}
\usepackage{amssymb}
\usepackage{enumitem}

\begin{document}

\mainmatter

%%%%%%%%%%%%%%%%%%%%%%%%%%%%%%%%%%%%%%%%%%%%%%%%%%%%%%%%%%%%%%%%%%%%%%%%%%%%%%%%%%%%%%
%%%%%%%%%%%%%%%%%%%%%%%%%%%%%%%%%%%%%%%%%%%%%%%%%%%%%%%%%%%%%%%%%%%%%%%%%%%%%%%%%%%%%%
%       Author's definitions
%%%%%%%%%%%%%%%%%%%%%%%%%%%%%%%%%%%%%%%%%%%%%%%%%%%%%%%%%%%%%%%%%%%%%%%%%%%%%%%%%%%%%%
%%%%%%%%%%%%%%%%%%%%%%%%%%%%%%%%%%%%%%%%%%%%%%%%%%%%%%%%%%%%%%%%%%%%%%%%%%%%%%%%%%%%%%

% Include here your \newcommand and \def definitions

%\newcommand {\RR} {\mathbb R}
%\newcommand {\NN} {\mathbb N}
%\newcommand {\TT} {\mathbb T}
%\newcommand {\SSs} {\mathbb S}
%\newcommand {\ZZ} {\mathbb Z}

%%%%%%%%%%%%%%%%%%%%%%%%%%%%%%%%%%%%%%%%%%%%%%%%%%%%%%%%%%%%%%%%%%%%%%%%%%%%%%%%%%%%%%
%%%%%%%%%%%%%%%%%%%%%%%%%%%%%%%%%%%%%%%%%%%%%%%%%%%%%%%%%%%%%%%%%%%%%%%%%%%%%%%%%%%%%%
%       Paper and author information
%%%%%%%%%%%%%%%%%%%%%%%%%%%%%%%%%%%%%%%%%%%%%%%%%%%%%%%%%%%%%%%%%%%%%%%%%%%%%%%%%%%%%%
%%%%%%%%%%%%%%%%%%%%%%%%%%%%%%%%%%%%%%%%%%%%%%%%%%%%%%%%%%%%%%%%%%%%%%%%%%%%%%%%%%%%%%

\author{Vladislav Babenko, Yuliya Babenko, Nataliya Parfinovych,
Dmytro Skorokhodov}

\title{Exact asymptotics of the optimal $L_{p}$-error of asymmetric linear spline approximation}

\affiliation{Vladislav Babenko,\\
Department of Mathematical Analysis and Theory of Functions \\
Dnepropetrovsk National University \\
pr. Gagarina, 72, \\
Dnepropetrovsk, UKRAINE, 49050 \\
{\tt  babenko.vladislav@gmail.com}\\ \\

Yuliya Babenko,\\
Department of Mathematics and Statistics\\
Kennesaw State University\\
1000 Chastain Road, $\#$1601\\
Atlanta, GA, USA 30144-5591\\
{\tt ybabenko@kennesaw.edu}\\ \\

Nataliya Parfinovych,\\
Department of Mathematical Analysis and Theory of Functions \\
Dnepropetrovsk National University \\
pr. Gagarina, 72, \\
Dnepropetrovsk, UKRAINE, 49050 \\
{\tt nparfinovich@yandex.ru}\\ \\

Dmytro Skorokhodov,\\
Department of Mathematical Analysis and Theory of Functions \\
Dnepropetrovsk National University \\
pr. Gagarina, 72, \\
Dnepropetrovsk, UKRAINE, 49050 \\
{\tt dmitriy.skorokhodov@gmail.com}
}

\thanks{This research was partially conducted during visit of V. Babenko, N. Parfinovych and D. Skorokhodov to Kennesaw State University (supported by QEP International Faculty Development grant and Simons Collaboration Grant Award ID 210363)}

\maketitle

% If it is necessary include here a short version for the title of the paper and for the list of authors.
% In that case remove also the % symbol comment mark of the following line.
\markboth{\em Babenko V., Babenko Y., Parfinovych N., Skorokhodov D.}{\em Optimal $L_{p}$-error of asymmetric linear spline approximation}

\abstract{In this paper we study the best asymmetric (sometimes also called penalized or sign-sensitive) approximation in the metrics of the space $L_p$, $1\leqslant p\leqslant\infty$, of functions $f\in C^2\left([0,1]^2\right)$ with nonnegative Hessian by piecewise linear splines $s\in S(\triangle_N)$, generated by given triangulations $\triangle_N$ with $N$ elements. We find the exact asymptotic behavior of optimal (over triangulations $\triangle_N$ and splines $s\in S(\triangle_N)$) error of such approximation as $N\to \infty$. }

\keywords{spline, asymmetric approximation, adaptive approximation, exact asymptotics, optimal error, anisotropic partitions, triangulations}

%  Put here your AMS Classifications instead of the ones of the example with
%  the format \classification{primary}{secondary}

\classificationps{41A15}{41A25, 41A60}

\section{Introduction} \label{Intro}

The question of approximation of functions defined on a polytope by piecewise polynomial functions (splines), generated with the help of a mesh (partition of the domain), in various metrics is of a great importance in Approximation Theory and its applications (numerical solutions for PDE's, surface simplification, image compression, terrain data processing etc.). 
Mainly, problems of approximation by interpolating splines have been considered. However, for various applications (as well as from theoretical point of view) problems of best and best one-sided approximation are important.
In Approximation Theory there exist two different tools to approximate functions: {\it uniform} methods, or methods that work rather well for all functions in a given class; and {\it adaptive} methods, or methods that take into account local variations (measured with the help of Hessian, curvature, their modifications etc.) in the behavior of each given function.
For methods which involve spline approximation, the adaptivity affects the construction of domain partitions and geometry (both size and shape) of its elements, which can be highly anisotropic.

In this paper we will consider the problem of adaptive approximation of twice differentiable functions by linear splines (naturally the domain partitions in this case are triangulations). This question has been studied extensively by many authors (see works~\cite{Nadler86, Daz3, huang, us, chen1, BBS, Cohen, JM} and references therein), but many interesting (from both theoretical and applied point of view) questions remain open.
%The question of finding optimal triangulations for each function is of particular interest for applications (see, for instance,~\cite{NiraDyn1, NiraDyn2, Cohen}). In order to construct adaptive triangulations many authors took into account the Hessian of the function (or curvature of its graph) (see~\cite{Nadler86, Daz3, huang, us, chen1, BBS, Cohen, JM}).  
In addition, note that the problem of surface approximation by linear splines is very close to an important problem in Geometry on approximation (in various metrics) of smooth convex bodies by various polytopes (inscribed, circumscribed, polytopes of best approximation etc.)
%On the other hand, one important problem in Geometry is to study the approximation (in a specified metric) of smooth convex bodies by various polytopes. For instance, questions of approximating convex bodies by inscribed or circumscribed polytopes, by polytopes with fixed number of vertices or faces, by polytopes of best approximation etc. have been studied in this direction. 
After some occasional results for bivariate functions, the book of L. Fejes~Toth~\cite{Toth} was the first to provide a large number of problems, ideas, and results on polytopal approximation in dimensions two and three, concentrating specifically on extremal properties of regular polytopes. Many extensions have been made afterwards to higher dimensions, other metrics etc. (see~\cite{Gr2, Bor, boro, Handbook} and references therein).  

In Approximation Theory there exists a tool to view both the problem of finding the best approximation without constraints and the problem of finding the best approximation with constraints ``under one umbrella''. This can be viewed as the best approximation in the spaces with asymmetric norm, or so-called ($\alpha,\beta$)-approximation (see, for example,~\cite{Bab82, Bab83, Korn}), when positive and negative parts of the difference between function and the approximant are ``weighted'' differently. Such type of approximations are of a separate interest since they can be considered as the problems of approximation with non-strict constraints (see below for more precise statements), when constraints are allowed to be violated, but the penalty for their violation is introduced into the error measure. Within this paper we will consider the questions of best $(\alpha,\beta)$-approximation by linear splines. We believe that such approach could be interesting and useful for some questions in Geometry as well.

Note that the construction of the best (in a specified sense) triangulation for approximation of an individual function,  or construction of the best polytope for an individual convex body, is an extremely difficult problem, and therefore it is natural to consider {\it asymptotically optimal sequences of polytopes} or {\it asymptotically optimal sequences of triangulations} (and splines defined on them). 

One possible method to construct asymptotically optimal sequences of triangulations or polytopes begins as follows. At the first step, we construct an intermediate approximation of the function (or convex body surface, respectively) by a piecewise quadratic function (surface). At the next step we approximate each quadratic piece by a piecewise linear function (spline) in the optimal way, generating a mesh of the domain. This in turn (at least in $\RR^2$) requires solving the following optimization problem (we will give its statement for approximation of functions in $\RR^d$). 

{\it Let a quadratic function $Q$ defined on $\RR^d$ be given. Consider the best approximation ($L_p$, asymmetric, one-sided) of $Q$ by linear functions on simplex $\mathcal{T}\subset\RR^d$ of unit volume. The problem is to find a simplex $\mathcal{T}^*$, for which the corresponding error is minimal. }
%(The known solutions of this problem are listed in Section~3.) 

This problem is important also in a number of questions of Geometry and Approximation Theory. 
%Therefore, in a number of questions of Geometry and Approximation Theory it is important to find a simplex of fixed (unit) volume such that the error of the best approximation of a given quadratic function on this simplex in a specified metric or the best approximation with constraints (for instance, one-sided) is minimized. 
In~\cite{simplex} we have proved the optimality of a regular simplex in the formulated problem for the best $(\alpha,\beta)$-approximation in $L_p$-metric of function $Q(\textbf{x}) = \sum\limits_{j=1}^{d} x_j^{2}$ by linear functions. Note that with the help of linear transformations the solution of this problem allows us to obtain the solution of analogous optimization problems for an arbitrary positive definite quadratic form. 

In this paper we will study the behavior of the optimal error of  ($\alpha,\beta$)-approximation in $L_p$-metric of an arbitrary $C^2$ function with nonnegative Hessian. The main contributions of this paper are:
\begin{enumerate}[topsep=0pt, partopsep=0pt, itemsep=2pt, parsep=2pt]
\item We present the construction of asymptotically optimal sequence of partitions and error estimates without assumption on the Hessian to be bounded away from zero. Remark that geometers were able to remove restrictions of such type in some of their problems before (see~\cite{Bor}). However, we use another technique to handle the problem.
%We remove the assumption on Hessian to be bounded away from zero. Geometers were able to remove this restriction in some of their problems before (see ~\cite{Bor}), and now we are presenting the construction of partitions (and, even more importantly, the error estimates) for the functions without this restriction as well. 
\item We impose no restrictions on triangulations (many existing works require some type of ``admissibility'').
\item We consider asymmetric approximation, which, as special cases, includes the cases of interpolating splines, splines of best approximation, splines of best one-sided approximation etc.
\end{enumerate}

The paper is organized as follows. In Section~\ref{Defs} we begin by introducing major concepts and definitions, in particular related to asymmetric approximation and asymptotically optimal triangulations. In Subsection~\ref{S2.3} we present the main questions that we will address in this paper, and state related geometric problems in Subsection~\ref{S3}. Subsection~\ref{S6} contains statements of the main results. In Section~\ref{ideas} we introduce the major ideas for the proofs of the main results without much of technical details. In particular, we relate the problem of describing exact asymptotics of the optimal error with some geometric problems. Statements of the solutions to these problems are presented in Section~\ref{S55}, which also contains additional geometric observations needed later for the lower estimate of the error. Section~\ref{S7} is dedicated to the construction of ``good'' triangulation for each fixed $N$ and the proof of estimate from above for the optimal error. The estimate from below is contained in Section~\ref{S8}.

\section {Notation, definitions, main questions, and results} \label{Defs}

%Let $\RR^2$ be the space of points in the plane endowed with the usual Euclidean distance ${\rm dist}(A,B)$, $A,B\in\RR^2$. 

Let the domain be $D:=[0,1]^2 \subset \RR^2$. We use this region for simplicity; the approach presented in this paper can be applied to any bounded connected region which is a finite union of triangles. By $C(D)$ we denote the space of functions continuous on $D$. Let $L_p:=L_{p}(D)$, $0< p\leqslant \infty$, be the space of measurable functions $f:D\to\RR$ such that $\|f\|_p < \infty$ where
$$
	\|f\|_p 
	= 
	\|f\|_{L_p(D)} 
	:= 
	\left\{
		\begin{array}{ll}
			\left(\displaystyle\int_{D} |f (x,y)|^p \,dx\,dy\right)^{1/p}, 
				&
			\textrm{if}\;\;\;    0< p < \infty, 
			\\ [10pt]
			\textrm{ess sup} \{|f (x,y) |\,:\,(x,y) \in D \}, 
				&
			\textrm{if}\;\;\;       p =\infty.
		\end{array}
	\right.
$$
For $p\geqslant 1$, $\|\cdot\|_p$ is the standard norm in space $L_p$. In addition, we use notation $\|\cdot\|_p$ when $p<1$ only for statement of main results. 

\subsection{Asymmetric approximation}\label{S2.1}
Let $f\in L_p$ and let $H$ be a subspace of $L_p$. By $E(f;H)_p$ we denote the best approximation of the function $f$ by the subspace $H$ in the $L_p$-metric, i.e.: 
$$
	E(f;H)_p
	:=
	E(f;H)_{L_p(D)}=\inf\{\|f-u\|_p\;:\;u\in H\}. 
$$
In addition, by 
$$
	E^{\pm}(f;H)_p
	:=
	E^{\pm}(f;H)_{L_p(D)}=\inf\{\|f-u\|_p\;:\;\pm u(x,y)\leqslant\pm f(x,y),\;(x,y)\in D\;\textrm{and}\;u\in H\}
$$
we denote the best one-sided approximation of the function $f$ by the subspace $H$ in the $L_p$-metric. In the case of $``+''$ in the above definition we say that we have {\it approximation from below}; in the case  of $``-''$ we say that we have {\it approximation from above}.

For $\alpha,\beta>0$ and $f\in L_p$, $1\leqslant p\leqslant\infty$, we define the {\it asymmetric ($\alpha,\beta$)-norm} as follows 
$$
	\|f\|_{p;\alpha,\beta}
	=
	\|f\|_{L_{p;\alpha,\beta}(D)}=\|\alpha f_++\beta f_-\|_{p},
$$
where $g_{\pm}(x,y)=\max\{\pm g(x,y);0\}$. Following the literature, we call $\|f\|_{p;\alpha,\beta}$ the {\it asymmetric norm}. Note that it satisfies the norm axioms except for the fact that we only have $\|\lambda f\|_{p;\alpha,\beta} = \lambda\|f\|_{p;\alpha,\beta}$ for $\lambda \geqslant 0$ (in particular, $\|f\|_{p;\alpha,\beta} \ne \|-f\|_{p;\alpha,\beta}$ for $\alpha\ne\beta$). Asymmetric norms in connection with various problems in Approximation Theory were considered in papers~\cite{Krein, Bab82, Dolzhenko1, Dolzhenko2} and books~\cite{KreinNudelman, Korn}.

By $E(f;H)_{p;\alpha,\beta}$ we denote the best ($\alpha,\beta$)-approximation~\cite{Bab82} of the function $f$ by the subspace $H$ in the $L_p$-metric, i.e.: 
$$
	E(f;H)_{p;\alpha,\beta}
	:=
	E(f;H)_{L_{p;\alpha,\beta}(D)}=\inf\{\|f-u\|_{p;\alpha,\beta}:\,u\in H\}.
$$ 
Note that for $\alpha=\beta=1$ we have $E(f;H)_{p;1,1}=E(f;H)_{p}$. V.~Babenko proved in~\cite{Bab82} that if $H\subset L_p(D)$, $1\leqslant p < \infty$, is locally compact, then for any $f\in L_p(D)$ the following limit relations hold true (see also~\cite{Korn}, Theorem~1.4.10): 
\begin{equation}
\label{ter}
	\lim_{\beta\to+\infty}E(f;H)_{p;1,\beta}
	=
	E^+(f;H)_{p}\qquad \textrm{and}\qquad \lim_{\alpha\to+\infty}E(f;H)_{p;\alpha,1}
	=
	E^-(f;H)_{p},
\end{equation} 
which are monotone in $\alpha$ and $\beta$. This allows us to include the problem of the best unconstrained approximation and the problem of the best one-sided approximation into the family of problems of the same type, and consider them from a general point of view (for more on this motivation, see~\cite{Bab83,Bab84}). In what follows we will allow the value $+\infty$ for $\alpha$ or $\beta$, in that case identifying $E(f;H)_{p;\alpha,\beta}$ with the corresponding one-sided approximation. Because of the relation 
$$
	\left\|f-u\right\|_{p;1,\beta}^p
	=
	\left\|f-u\right\|_{p}^p + (\beta^p-1)\left\|(f-u)_-\right\|_p^p, 
		\qquad \beta>1, 
$$
the problem of the best $(1,\beta)$-approximation can be considered as the problem of the best approximation with non-strict constraint $f(x,y)\leqslant u(x,y)$, $(x,y)\in D$. This constraint is allowed to be violated, but the penalty 
$$
	(\beta^p-1)\left\|(f-u)_-\right\|_{p}^p
$$
for the violation is introduced into the error measure. In what follows we will allow the value $+\infty$ for $\alpha$ or $\beta$, in that case identifying $E(f;H)_{p;\alpha,\beta}$ with the corresponding one-sided approximation. 

\subsection{Optimal triangulations and asymptotically optimal sequences of triangulations} \label{S2.2}

Let $N\in\NN$. A collection $\triangle_N=\triangle_N(D)=\{T_i\}^N_{i=1}$ of $N$ triangles in the plane is called a {\it triangulation} of the set $D$ provided that

$\left.1\right)$ any pair of triangles from $\triangle_N$ intersect at most at a common vertex or along a common edge; 

$\left.2\right)$ $D=\displaystyle \bigcup\limits_{i=1}^N{T_i}$.

Let $\mathcal{P}_1$ be the set of bivariate linear polynomials $p(x,y)=ax+by+c$, with $a,b,c\in\RR$. Given a triangulation $\triangle_N$, define the class of linear splines $\mathcal{S}(\triangle_N)$ as follows 
$$
	\mathcal{S}(\triangle_N)
	:=
	\left\{ 
		f \in C(D): 
		\forall i=1,...,N\;\; 
		\exists p_i\in \mathcal{P}_1\;
		\textrm{such\;that}\;
		f|_{T_i}=p_i|_{T_i}
	\right\}.
$$

Now let the function $f\in C^2(D)$ and the number $N$ of triangles be fixed. Set
$$
	R_N(f,L_{p; \alpha, \beta})
	:=
	\inf_{\triangle_N} E(f;\mathcal{S}(\triangle_N))_{p; \alpha, \beta} 
	=
	\displaystyle \inf_{\triangle_N}\inf_{s\in\mathcal{S}(\triangle_N)}\|f-s\|_{p; \alpha, \beta}.
$$
This quantity we will call the {\it optimal} $L_p$-error of piecewise linear ($\alpha,\beta$)-approximation of the function $f$ on triangulations with $N$ elements. A triangulation $\triangle_N^0$ and the corresponding spline $s^0_N\in\mathcal{S}(\triangle_N^0)$ are called {\it $L_{p;\alpha,\beta}$-optimal} for the given function $f$ if 
$$
	\left\|f-s^0_N\right\|_{p; \alpha , \beta }
	=
	R_N(f,L_{p; \alpha , \beta }).
$$  

Note that the quantity $R_N(f,L_{p;1,1})$ coincides with the error of the best $L_p$-approximation of function $f$ by linear splines from $\mathcal{S}(\triangle_N)$. In addition, in view of~(\ref{ter}) in the case $\alpha=1$ and $\beta\to\infty$ ($\beta=1$ and $\alpha\to\infty$) the quantity $R_N(f,L_{p;\alpha,\beta})$ tends to the error of the best $L_p$-approximation of function $f$ from above (below) by splines from $\mathcal{S}(\triangle_N)$. Remark that the latest statement does not immediately follow from~(\ref{ter}). However, the proof of it is rather simple and we omit it here.

%{\bf Remark~1.} The last statement does not immediately follow from (\ref{ter}). However, the proof of it is rather simple.
%Indeed, it can be easily seen that 
%$$
%  \lim\limits_{\beta\to\infty}\inf\limits_{\triangle_N} E(f;\mathcal{S}(\triangle_N))_{p;1,\beta} \le \inf\limits_{\triangle_N}\lim\limits_{\beta\to\infty} E(f;\mathcal{S}(\triangle_N))_{p;1,\beta}. 
%$$
%On the other hand, for every $\varepsilon>0$ and $\beta>0$, there exists a triangulation, denote it by $\triangle_{N;\beta}$, such that $\inf\limits_{\triangle_N} E(f;\mathcal{S}(\triangle_N))_{p;1,\beta}\ge E(f;\mathcal{S}(\triangle_{N;\beta}))_{p;1,\beta}-\varepsilon$. In addition, by $\triangle_N^*$ let us denote the limit triangulation of some convergent subsequence $\{\triangle_{N;\beta_k}\}_{k=1}^{\infty}$ of the family $\{\triangle_{N;\beta}\}_{\beta>0}$. Hence, due to the continuous dependence of the error of the best asymmetric approximation on triangulation, we have 
%$$
%  \lim\limits_{\beta\to\infty}\inf\limits_{\triangle_N} E(f;\mathcal{S}(\triangle_N))_{p;1,\beta} \ge \lim\limits_{\beta\to\infty} E(f;\mathcal{S}(\triangle_N^{*}))_{p;1,\beta} - 2\varepsilon, 
%$$
%which, combined together with relations~(\ref{ter}), shows the convergence of $R_{N}(f,L_{p;1,\beta})$ as $\beta \to \infty$. In particular, when the function $f$ is convex we obtain that $R_N(f,L_{p;1,\beta})$ tends (as $\beta \to \infty$) to the optimal (over $\triangle_N$) $L_p$-error of interpolation of $f$ by splines from $\mathcal{S}(\triangle_N)$. 

\subsection{Main questions} \label{S2.3}
In this paper we will work with functions $f\in C^2(D)$, where $D$ for simplicity is taken to be a unit square $[0,1]^2\subset \RR^2$.
It is well known that for most of such functions (i.e. for functions with the Hessian not identically equal to zero) the order of the optimal error $R_N(f,L_{p; \alpha, \beta})$ is $\frac 1N$ as $N\to \infty$. Our goal is to study sharp asymptotic behavior of the quantity $R_N(f,L_{p; \alpha, \beta})$ as $N\to \infty$. To that end, we will prove the existence of the limit of $N\cdot R_N(f,L_{p; \alpha, \beta})$ as $N\to \infty$, and will find its exact value. In turn, our analysis will allow one to obtain information about construction of asymptotically optimal sequences of triangulations.

We will study this problem in the case when the given function has nonnegative Hessian. In the case when the Hessian is strictly positive there exist a lot of results (see works~\cite{Nadler86, Daz3, us, chen1, BBS, Cohen, JM} and references therein). One of the most studied questions is the problem of approximation by linear interpolating splines. Note that in the case when Hessian is strictly positive interpolating splines obviously coincide with splines of best (one-sided) approximation from above. However, even for the case of strictly positive Hessian the questions of finding sharp asymptotics of the optimal error in the cases of approximating by splines of best approximation, best approximation from below, and, in general, best asymmetric approximation remain open. Finding solutions to these questions is one of two main goals of the present paper. Besides that, in all these questions we will remove the restriction of Hessian being bounded away from zero to allow algorithms to be applicable for wider range of surfaces. In~\cite{Bor} B\"or\"oczky addressed this nontrivial question in the related case of approximating smooth convex bodies by some inscribed polytopes. 

As for the little investigated case of negative Hessian, the only known (at least to us) explicit result is~\cite{us}.

\subsection{Related geometric problems} \label{S3}

The essential role in further results is played by the solution of the following extremal problems. Let $q(x,y):= x^2+y^2$. 

{
\bf Problem~1. 
\it For $\alpha,\beta>0$ and $1\leqslant p\leqslant\infty$, find 
\begin{equation}
\label{Cp}
 	C_{p;\alpha,\beta}
 	:=
 	\displaystyle \inf_{T} \frac{E(q;\mathcal{P}_1)_{L_{p;\alpha,\beta}(T)}}{|T| ^{1+1/p}},
\end{equation}
where the infimum is taken over all triangles $T$ in $\RR^2$ and $|T|$ stands for the area of triangle $T$.
}

Solution to this problem allows solving a similar problem for arbitrary positive definite quadratic form (see Section~\ref{S4} for details). 

Note that Problem~1 is a generalization of the following problems. 

{
\bf Problem~2. 
\it For $1\leqslant p\leqslant\infty$, find 
\begin{equation}
\label{Cpprosto}
	C_{p}
	:=
	\displaystyle \inf_{T} \frac{E(q;\mathcal{P}_1)_{L_{p}(T)}}{|T| ^{1+1/p}} 
	\qquad\textrm{and}\qquad 
	C_{p}^{\pm}
	:=
	\displaystyle \inf_{T} \frac{E^{\pm}(q,\mathcal{P}_1)_{L_{p}(T)}}{|T| ^{1+1/p}}.
\end{equation}
}

The constant $C_p^-$ coincides with the best $L_p$-error of interpolation of $q$ by linear functions over triangles of unit area.

To the best of our knowledge the progress on the problem of computing the constant $C_p^-$ (see Problem~3 above) can be outlined as follows: 
\begin{enumerate}[topsep=0pt, partopsep=0pt, itemsep=2pt, parsep=2pt]
\item $p=\infty$ (D'Azevedo and Simpson~\cite{Daz3}, 1989); 
\item $p=1$ (B\"or\"oczky, Ludwig~\cite{boro}, 1999); 
\item $p=2$ (Pottmann et al~\cite{kodla1}, 2000);
\item $p\in\NN$ (Chen~\cite{chen1}, 2007);
\item $p\in(1,\infty)$ (V.~Babenko, Yu.~Babenko, and Skorokhodov~\cite{BBS}, and independently Chen (\cite{chen}, 2008) for any dimension $d$).
\end{enumerate} 

The most general constant $C_{p;\alpha,\beta}$ for any dimension $d\in \NN$ was found by the authors in~\cite{simplex}. Note that the infimum in the definition of constant $C_{p;\alpha,\beta}$ is achieved only on regular simplices. 

%The rest of the paper is organized as follows. Subsection 1.3 contains formal statements of the main results. In Section 2 it will be shown that the infimum in the definition of $C_{p;\alpha,\beta}$ for every $\alpha,\beta>0$ is attained on equilateral triangles. The proof of this fact is listed in the present paper only for the sake of completleness. Section 3 provides certain preliminary results, in particular, on how affine transformations affect the error of asymmetric approximation of a quadratic function by linear splines. The proof of the upper estimate in Theorem~2 for the case $1\le p<\infty$ will be provided in Section~4. Section 5 contains the proof of the lower estimate in Theorem~2 for the case $1\le p<\infty$. In Section~6 we will prove Theorem~2 for the case $p=\infty$. In the last section of the paper we finalize the construction of asymptotically optimal sequence of triangulations $\{\triangle_N^*\}_{N=1}^{\infty}$ and splines $\{s_N^*\}_{N=1}^{\infty}$, $s_N^*\in\mathcal{S}(\triangle_N^*)$.  

{\bf Remark~1.} Note that the $\inf$ in all the constants $C_{p;\alpha,\beta}$, $C_p$, and $C^{\pm}_p$, that are solutions of Problems 1, 2 and 3 posed in the previous section, are achieved on regular triangles. This fact for $C_{p; \alpha, \beta}$ was proved in greater generality (for any dimension $d$) in~\cite{simplex}. 
%In the present paper, in particular, we present for completeness an alternative, simpler proof of this fact in the case $d=2$.

\subsection{Main results} \label{S6}

% In order to simplify its statement and further proof, for every $p\in (0,1)$, we set 
%$$
%  \left\|f\right\|_{p} = \left(\int_D |f(x,y)|^p\,dx\,dy\right)^{1/p}. 
%$$
%\subsection{Main results} \label{S3.1}

The following theorem is the main result of this paper.

\begin{theorem}\label{Th2}
Let $f \in C^2(D)$ be such that $H(f;x,y)\geqslant 0$ for all $(x,y) \in D$. Then for all $\alpha,\beta>0$ and $1\leqslant p\leqslant\infty$,
\begin{equation}
\label{T2}
	\lim_{N \to \infty} N \cdot R_N(f,L_{p; \alpha , \beta })
	= 
	2^{-1} C_{p;\alpha,\beta}\cdot\left\|\sqrt{H}\right\|_{\frac{p}{p+1}}, 
\end{equation}
where $C_{p;\alpha,\beta}$ was defined in~(\ref{Cp}). 
\end{theorem}

For regular (no constraints) best and best one-sided $L_p$-approximations, we obtain the following corollaries. 

\begin{corollary}
Let $f \in C^2(D)$, $H(f;x,y)\geqslant 0$ for all $(x,y) \in D$. Then for all $1\leqslant p\leqslant\infty$,
$$
	\lim_{N \to \infty} N \cdot R_N(f,L_{p})
	= 
	2^{-1}C_p \cdot\left\|\sqrt{H}\right\|_{\frac{p}{p+1}},
$$
$$
	\lim_{N \to \infty} N \cdot \inf\limits_{\triangle_N} E^{\pm}(f;S(\triangle_N))_p 
  	= 
  	2^{-1}C_p^{\pm}\cdot\left\|\sqrt{H}\right\|_{\frac{p}{p+1}},
$$
where $C_p$ and $C^{\pm}_p$ were defined in~(\ref{Cpprosto}). 
\end{corollary}

%\begin{corollary}
%Let $f \in C^2(D)$, $H(f;x,y)\ge 0$ for all $(x,y) \in D$. Then for all $1\le p\le\infty$, 
%where  were defined in~(\ref{Cppm}). 
%\end{corollary}

{\bf Remark~2.} Corollary~2.2 generalizes Theorem~2 in paper~\cite{BBS} (case of interpolating splines). 

%{\bf Remark~3.} From the proof of the main Theorem~\ref{Th2} the explicit algorithmic constructions of triangulations will be clear. 

{\bf Remark~3.} Assertion of Theorem~\ref{Th2} remains true if we replace the space $\mathcal{S}\left(\Delta\right)$ of continuous piecewise linear splines on triangulation $\Delta$ by wider space $\overline{\mathcal{S}}\left(\Delta\right)$ of arbitrary piecewise linear on elements from triangulation $\Delta$ splines which are not necessarily continuous.

%{\bf Remark~4.} Note that we can consider the problems of asymmetric approximation in a more general setting by introducing two positive continuous functions $\alpha\,:D\,\to\RR$ and $\beta\,:\,D\to\RR$ instead of constants $\alpha$ and $\beta$. %Similarly to further presented proofs can be used to derive the following generalization of Theorem~2. 
%The proof presented in this paper can be extended to establish the following theorem. 

%\begin{theorem}\label{T3}
%Let $f \in C^2(D)$, $H(f;x,y)\ge 0$ for all $(x,y) \in D$, and let $\alpha,\beta$ be positive continuous on $D$ functions. Then for all $1\le p\le\infty$,
%$$
%\lim_{N \to \infty} N \cdot R_N\left(f,L_{p; \alpha(\cdot), \beta(\cdot)}\right)
%  = 2^{-1} \left\|\sqrt{H} C_{p;\alpha(\cdot),\beta(\cdot)}\right\|_{\frac{p}{p+1}}.%\left( \displaystyle \int_{D}H^{\frac{p}{2(p+1)}}(f;x,y) C_{p;\alpha(x,y),\beta(x,y)}^{\frac{p}{p+1}}\,dxdy\right)^{\frac{p+1}{p}}. 
%$$
%\end{theorem}

\section{Ideas used in the proof of the main result}\label{ideas}

In order to make the reading of the rest of the paper easier we would like to devote this section to introduce the main ideas without technical details. 

Let $f\in C^2(D)$ be such that $H(f; x,y)\geqslant 0$ for all $(x,y)\in D$. To prove Theorem~\ref{Th2} %, i.e. limit relation (\ref{T2}), 
we will show that %there exists a constant $K>0$ such that 
for any $\varepsilon>0$
\begin{equation}
\label{above}
	\limsup_{N \to \infty} N \cdot R_N(f,L_{p{; \alpha , \beta }}) 
	\leqslant
	2^{-1}C_{p;\alpha,\beta}\cdot\left\|\sqrt{H}\right\|_{\frac{p}{p+1}}(1+\varepsilon), 
\end{equation}
and
\begin{equation}
\label{below}
	\liminf_{N \to \infty} N \cdot R_N(f,L_{p; \alpha , \beta })
	\geqslant 
	2^{-1}C_{p;\alpha,\beta}\cdot\left\|\sqrt{H}\right\|_{\frac{p}{p+1}}(1-\varepsilon).
\end{equation}
Further, we will refer to the proof of inequality~(\ref{above}) as ``estimate from above'', and to inequality~(\ref{below}) -- as ``estimate from below''.

Note that the following ideas in the case of strictly positive Hessian and for the case of best approximation from above (interpolation) have already been introduced in~\cite{us} and~\cite{BBS}.

\subsection{Main ideas used in the proof of estimate from above}
\label{S3.2}

To prove inequality~(\ref{above}), for every $\varepsilon >0$, we will present a particular sequence of triangulations $\{ \Delta_N(\varepsilon)\}_{N=1}^{\infty}$ and corresponding sequence of splines $\{ s_N(\varepsilon)\}_{N=1}^{\infty}$, $s_N(\varepsilon)\in S(\Delta_N(\varepsilon))$, such that
\begin{equation}
\label{above2}
	\limsup_{N \to \infty} N \cdot \|f-s_N(\varepsilon)\|_{L_{p; \alpha, \beta}(D)}
	\leqslant 
	2^{-1}C_{p}\cdot\left\|\sqrt{H}\right\|_{\frac{p}{p+1}}\cdot (1+\varepsilon).
\end{equation}
%{First} we split $D$ into smaller squares $D_i^N$, $i=1,\ldots,m_N^2$, where $m_N^2\to\infty$ as $N\to\infty$ and $m_N^2=o(N)$ as $N\to\infty$ (more precisely $m_N$ will be defined in the next section). {Then} on every square $D_i^N$ we replace $f$ by its Taylor polynomial $f_{N;i}$ of second order constructed at some point inside $D_i^N$ (we take its center for simplicity).

We begin by following the ideas developed in~\cite{us, BBS}. For each $N$ we will construct a triangulation based on the idea of intermediate approximation of functions from the class $C^2(D)$ by piecewise quadratic functions. 

First, for $\delta>0$, define the {\it modulus of continuity} of $g\in C(D)$ by
\begin{equation}
\label{m1}
	\omega\left(g,\delta\right) 
	:= 
	\sup\limits\left\{\left|g(x') - g(x'')\right|\,:\,\left(x',y'\right)\in D,\,\left|x'-x''\right|\leqslant \delta, \, \left|y'-y''\right|\leqslant \delta\right\},
\end{equation}
and for $(x,y)\in D$, consider 
$$
	\lambda_{\min}(x,y) 
	:= 
	\left(f_{xx}(x,y) + f_{yy}(x,y)\right)/4 - \sqrt{\left(f_{xx}(x,y) + f_{yy}(x,y)\right)^2/16 - H(f;x,y)/4}
$$

We begin by splitting the domain $D$ into small in comparison with $N$ ($o(N)$ as $N\to \infty$) number $m_N^2$ of subdomains ($m_N^2\to \infty$ as $N\to \infty$). In the case $D=[0,1]^2$ we will take these subdomains to be squares (for simplicity). Denote them by $D^N_i$, $i=1,\ldots, m_N^2$. On each $D^N_i$ we will consider intermediate approximation of $f$ by $f_{N,i}$ that is second degree Taylor polynomial of $f$ constructed at some point (we take center) inside of $D^N_i$. Observe that the order of this approximation is higher than the order of the optimal error $R_N(f,L_{p;\alpha,\beta})$ and therefore will not affect the main term of the asymptotics of $R_N(f,L_{p;\alpha,\beta})$.  In addition, note that to approximate quadratic function $f_{N,i}$ by linear splines is the same as to approximate its quadratic part (denote it by $Q_{N,i}$) by linear splines. Hence, instead of function $f\in C^2(D)$ we will be approximating $Q_{N,i}$ by linear splines on every square $D_i^N$. This switch to an intermediate approximation is one of the major ideas in obtaining asymptotically optimal sequences of triangulations and exact asymptotics of the optimal error $R_N(f,L_{p;\alpha,\beta})$.

The coefficients of quadratic form $Q_{N,i}$, which obviously depend on $f$, determine the geometry of optimal mesh element (triangle) on the particular subdomain $D^N_i$ as follows.  We begin by finding eigenvalues $\lambda_{\min}^{N,i}$ and $\lambda_{\max}^{N,i}$ of the quadratic form $Q_{N,i}$ for each $i=1,\ldots,m_N^2$. Depending on the magnitude of the eigenvalues on each subdomain, we will split all the squares $D^N_i$ into four groups. 
\begin{enumerate}[topsep=0pt, partopsep=0pt, itemsep=2pt, parsep=2pt]
\item The first group contains squares with  $\lambda_{\max}^{N,i}\geqslant \lambda_{\min}^{N,i}>\varepsilon$.

In this case, the intermediate approximation $Q_{N;i}$ of $f$ on $D^N_i$ is an elliptic paraboloid. We find the optimal triangle (its shape and orientation in the plane) by solving a local optimization problem of minimizing ${L_{p}}$-error of approximation of $Q_{N;i}$ by linear functions. For the details of the solution see Sections \ref{S4} and \ref{S5}. We then choose the size of the triangles so that their amount is approximately $N$ and overall error of approximation on triangles from the squares in this group is minimized.
Note that the error on squares from the first group will be a main contribution to the global error and main focus of our attention.

Let $T_i^N$ be an optimal triangle which solves the minimization problem on $D_i^N$. {We} use $T_i^N$, together with its reflections and translations, to provide a triangulation of each $D_i^N$ with relatively small ($o\left(m_N^{-2}N\right)$ as $N\to\infty$) number of other triangles (created by refining shapes along the boundary {of $D_i^N$}).

\item The second group contains squares with  $\omega\left(\lambda_{\min},m_N^{-1}\right)<\lambda_{{\min}}^{N,i}\leqslant \varepsilon$.

There will be very few squares from this group and its contribution to the global error will not be significant. The squares from this group we subdivide into equal right isosceles triangles in the amount of $o\left(m_N^{-2}N\right)$.

\item  The third group contains squares with  $\lambda_{\min}^{N,i}\leqslant \omega\left(\lambda_{\min},m_N^{-1}\right)$ and $\lambda_{\max}^{N,i}\geqslant \varepsilon$.

In this case, the intermediate approximation $Q_{N;i}$ on $D^N_i$ is close to a parabolic cylinder. The squares $D^N_i$ from this group will be divided into $\varepsilon \cdot o\left(m_N^{-2}N\right)$ of right triangles with the longer side positioned in the direction of eigenvector corresponding to the smallest eigenvalue. 

\item The fourth group contains squares with  $\lambda_{\min}^{N,i}\leqslant \omega\left(\lambda_{\min},m_N^{-1}\right)$ and $\lambda_{\max}^{N,i}\leqslant\varepsilon$.

The intermediate approximation is almost a plane. The squares from this group we subdivide into equal right isosceles triangles in the amount of $\varepsilon \cdot o\left(m_N^{-2}N\right)$. 

\end{enumerate}

``Gluing'' (without adding new vertices) triangulations of all $D_i^N$, we obtain the desired triangulation ${\triangle_N (\varepsilon)}$ of $D$. 
 
Having the triangulation ${\triangle_N (\varepsilon)}$ of the whole domain $D$, we will then define the spline $s_N(\varepsilon;x,y)\in S(\Delta_N(\varepsilon))$ which approximate $f$ sufficiently well. First of all, on the union of all triangles (denoted by $U_N$), each of which is contained in the interior of the corresponding square $D_i^N$ from the first group, we set $s_N(\varepsilon;x,y) := \tilde{s}_N(\varepsilon;x,y)$, where $\tilde{s}_N$ is the spline of the best $L_{p;\alpha,\beta}$-approximation (on $U_N$) of respective $Q_{N;i}$. On triangles that are contained in $D\setminus\textrm{int}\,\left(U_N\right)$, we let $s_N(\varepsilon;x,y)$ to interpolate function $f$ at the vertices of $\Delta_N(\varepsilon)$ which are located in $D\setminus U_N$, and interpolate $\tilde{s}_N$ at points on the boundary of $U_N$.  
% define a spline  to be the spline of best $L_{p;\alpha,\beta}$-approximation (on this union) of respective $Q_{N;i}$. Then we will define a spline $\tilde{\tilde s}_N$ on the set $D\setminus \textrm{int}\,(U_N)$ as the linear spline that interpolates the function $f$ at those vertices of ${\triangle_N (\varepsilon)}$ that lie in $D\setminus U_N$ and interpolates spline $\tilde s_N$ on the vertices that lie on the boundary $\partial U_N$. Finally, we define the desired spline $s_N$ as follows
% $$
%s_N(\varepsilon; x,y):=\left\{\begin{array}{ll}
%              \tilde s_N(x,y),   &{\rm if}\;\;\;    (x,y)\in \textrm{int}\,U_N, \\ [3pt]
%                     \tilde{\tilde s}_N(x,y),                           &\textrm{otherwise}.
%\end{array}\right.
%$$
Note that the spline defined in such a way is continuous on the whole domain $D$.

The constructed sequence of triangulations $\{ \Delta_N(\varepsilon)\}_{N=1}^{\infty}$ and corresponding sequence of splines $\{ s_N(\varepsilon)\}_{N=1}^{\infty}$ will allow us to prove the estimate from above (\ref{above2}).

\subsection{Main ideas used in the estimate from below}\label{S3.3}

To prove estimate~(\ref{below}) we will show that for every (sufficiently small) number $\varepsilon>0$ and every sequence $\left\{\triangle_N\right\}_{N=1}^{\infty}$ of triangulations with property 
$$
	\liminf\limits_{N\to \infty} N\cdot \inf\limits_{s\in\mathcal{S}\left(\triangle_N\right)} \|f-s\|_{L_{p; \alpha, \beta}(D)} 
	<
	\infty, 
$$
the following inequality holds true
\begin{equation}
\label{lower_sequentia}
	\liminf\limits_{N\to\infty} N \cdot \inf\limits_{s\in\mathcal{S}\left(\triangle_N\right)} \|f-s\|_{L_{p; \alpha, \beta}(D)} 
	\geqslant 
	2^{-1} C_p \left\|\sqrt{H}\right\|_{\frac{p}{p+1}} (1-\varepsilon). 
\end{equation}

The general idea of the proof of inequality~(\ref{lower_sequentia}) is to classify the triangles from triangulation $\triangle_N$ into two categories: ``good'' triangles and ``bad'' triangles. First we show that the errors $\inf\limits_{P\in\mathcal{P}_1} \|f-P\|_{L_{p;\alpha,\beta}(T)}$ for every ``bad'' triangle $T$ can be neglected. According to this observation, we can study the errors $\inf\limits_{P\in\mathcal{P}_1} \|f-P\|_{L_{p;\alpha,\beta}(T)}$ only for ``good'' triangles $T\in\triangle_N$. For such triangles $T$ we use the idea of intermediate approximation and substitute the function $f$ by its second degree Taylor polynomial $f_{N;T}$ constructed at an arbitrary point inside $T$. Then we use results of Sections~\ref{S4} and~\ref{S5} together with the Jensen inequality to obtain the desired inequality~(\ref{lower_sequentia}). 

Next, we clarify the classification of triangles $T\in\triangle_N$ into ``good'' and ``bad''. To this end the triangles of each triangulation $\triangle_N$, $N\in\NN$, we divide into five groups according to the following rules. Assume that $T\in\triangle_N$. Then: 
\begin{enumerate}[topsep=0pt, partopsep=0pt, itemsep=2pt, parsep=2pt]
\item $T\in A_1^N$ iff $H(f;x,y)<2\varepsilon$ at every point $(x,y)\in T$; 
\item $T\in A_2^N$ iff $T\not\in A_1^N$, $H(f;x,y)\geqslant\varepsilon$ at every point $(x,y)\in T$ and $\left\|f-f_{N;T}\right\|_{L_{p;\alpha,\beta}(T)}$ is significantly lower than $\inf\limits_{P\in\mathcal{P}_1}\|f-P\|_{L_{p;\alpha,\beta}(T)}$;
\item $T\in A_3^N$ iff $T\not \in A_1^N$, $H(f;x,y)>\varepsilon$ at every point $(x,y)\in T$ and $\textrm{diam}\,T$ is very large; 
\item $T\in A_4^N$ iff $T\not\in A_1^N\cup A_2^N\cup A_3^N$ and $H(f;x,y)>\varepsilon$ at every point $(x,y)\in T$; 
\item $T\in A_5^N$ iff there exist two points $(x',y')\in T$ and $(x'',y'')\in T$ such that $H(f;x',y')<\varepsilon$ and $H(f;x'',y'')\geqslant 2\varepsilon$. 
\end{enumerate}
In Lemmas~\ref{L} and~\ref{L10} we will show that the overall area of triangles belonging to sets $A_3^N$, $A_4^N$ and $A_5^N$ is less than $\varepsilon$. Therefore, this fact and the definition of the group $A_1^N$, allow us to classify the triangles $T\in A_1^N\cup A_3^N\cup A_4^N \cup A_5^N$ as ``bad'' and the triangles $T\in A_2^N$ as ``good'' respectively.

\section{Construction of an optimal mesh element}
\label{S55}

\subsection{Optimality of a regular triangle for $E(x^2+y^2;\mathcal{P}_1)_{L_{p;\alpha,\beta}}$}
\label{S4}

Here we would like to state the solution of Problem~1 posed in Section~\ref{S3}. For the proof of this result we refer the reader to~\cite{simplex}. 

%From now on, let $T_0$ denote the equilateral triangle of unit area. 

\begin{theorem}
\label{T1}
Let $q(x,y)=x^2+y^2$. Then for every $\alpha,\beta>0$ and $1\leqslant p\leqslant\infty$, 
$$
	C_{p;\alpha,\beta} 
	= 
	E(q;\mathcal{P}_1)_{L_{p;\alpha,\beta}(T_0)},
$$
where $C_{p;\alpha,\beta}$ was defined in~(\ref{Cp}) and $T_0$ is equilateral triangle of unit area. 
\end{theorem}

This result can be proved in a significantly simpler and more elegant way (comparing to the general result for any dimension $d$ in~\cite{simplex}) using the idea of symmetry and averaging, which was also used in~\cite{BBS}.

{\bf Remark~4.} Let $P$ be the polynomial of the best approximation of $q$ on equilateral triangle $T$. Using arguments about symmetry and rotational invariance we can conclude that the difference $q-P$ attains the same values at three vertices of $T_0$.

In certain cases, the constant $C_{p;\alpha,\beta}$ can be found explicitly. For instance, 
$$
	C_{\infty;\alpha,\beta} 
	= 
	4\cdot 3^{-3/2}\alpha\beta(\alpha+\beta)^{-1}, 
$$
and in the case $3^{3/2}\pi^{-1}\alpha\leqslant \alpha+\beta$,
$$
	C_{1; \alpha, \beta} 
	= 
	3^{-3/2}\alpha- 2^{-1}\pi^{-1}\alpha^2(\alpha+\beta)^{-1}.
$$

\subsection{Geometry of optimal triangle for $E(Ax^2+By^2+2Cxy; \mathcal{P}_1)_{L_{p;\alpha,\beta}}$} 
\label{S5}

In previous section we found that the optimal triangle for approximation of the form $q(x,y)=x^2+y^2$ is the equilateral triangle. Let us consider general positive definite quadratic form $Q(x,y)=Ax^2+By^2+2Cxy$ (i.e. $AB > C^2$) and find a unit area triangle $T$ that delivers the infimum in the problem 
\begin{equation}
\label{Pr1}
	E\left(Q;\mathcal{P}_1\right)_{L_{p;\alpha,\beta}(T)} \to \inf\limits_{T}.
\end{equation}
Let $\lambda_1,\lambda_2$ be the eigenvalues of the matrix 
$
	S 
	= 
	\left(
		\begin{array}{cc}
			A 
				& 
			C 
			\\ 
			C 
				& 
			B
		\end{array}
	\right)
$, and by $U$ we denote $2\times 2$ matrix composed from eigenvectors of $S$ having unit length. Then the linear mapping 
\begin{equation}
\label{direct_tr}
	\left(
		\begin{array}{l} 
			x 
			\\ 
			y
		\end{array}
	\right) 
	= 
	U \left(
		\begin{array}{ll} 
			\lambda_1^{-1/2} 
				& 
			0 
			\\ 
			0 
				& 
			\lambda_2^{-1/2}
		\end{array}
	\right)
	\left(
		\begin{array}{l} 
			u 
			\\ 
			v
		\end{array}
	\right)
\end{equation}
transforms quadratic form $Q(x,y)$ into the form $q(u,v)$. Hence, optimal triangle $T$ for problem~(\ref{Pr1}) is obtained from the equilateral triangle by applying the inverse transformation to~(\ref{direct_tr}). Therefore, we established the following fact.

%\textcolor{red}{and study how it} affect s the geometry of an optimal triangle. In particular, we will prove the following corollary of Theorem \ref{T1} (the proof will follow from Lemmas \ref{L5} and \ref{L6} below).

\begin{corollary}
\label{cor}
Let $Q(x,y)=Ax^2+By^2+2Cxy$ be a positive definite quadratic form, i.e. $AB-C^2>0$. Then for every $\alpha,\beta>0$ and $1\leqslant p\leqslant\infty$, 
\begin{equation}
\label{Cp1} 
	\inf_{T}\frac{E(Q;\mathcal{P}_1)_{L_{p;\alpha,\beta}(T)}}{|T|^{1+1/p}} 
	=
	C_{p;\alpha,\beta}\sqrt{AB-C^2} 
	=
	E(q;\mathcal{P}_1)_{L_{p;\alpha,\beta}(T_0)}\sqrt{AB-C^2},
\end{equation}
where, as defined above, $T_0$ is a regular triangle of unit area.
\end{corollary}

{\bf Remark~8.} Let triangle $T$ deliver the infimum in problem~(\ref{Pr1}) for positively definite quadratic form $Q$. If $P$ is the linear polynomial of the best approximation of $Q$ on $T$ then the difference $Q-P$ attains equal values at three vertices of $T$.

In addition to Corollary~\ref{cor}, we need the following two lemmas.

\begin{lemma}
\label{L7}
Let us consider the collection of quadratic forms $Ax^2 + By^2 + 2Cxy$ which satisfy conditions $0<A\leqslant A^+$, $0<B\leqslant B^+$ and $H=AB-C^2 \geqslant K$, where $A^+$, $B^+$, $K$ are some positive numbers. Then for any such form $\lambda_{\min} \geqslant 
\sqrt{K}>0$.
\end{lemma}

Proof trivially follows from the fact that the function $g(u,v)=u-\sqrt{u^2-v}$ ($u>0$, $0<v\leqslant 1$) is decreasing in $u$ and is increasing in $v$.

\begin{lemma}
\label{L8}
For the collection of quadratic forms satisfying the assumptions of Lemma~\ref{L7}, the ratio of the diameter of the optimal triangle to the square root of the area of this triangle is bounded above by the constant independent of $A^+$, $B^+$, and $K$.
\end{lemma}

This statement follows from Lemma~\ref{L7}. % The next result will be used in Section~5. 

\subsection{Additional geometric observations}

In paper~\cite{BBS} the following lemma was proved. 

\begin{lemma}
\label{L9}
Let $f\in C^2(D)$; $H(f;x,y)\geqslant K>0$ for all $(x,y)\in D$. If $\bar{n}$ is an arbitrary unit vector in the plane, then
\begin{equation}
\label{const}
	\left|\frac{\partial^2 f}{\partial\bar{n}^2}\right|
	\geqslant 
	\frac{K}{2} \min{\left\{\frac 1{\|f_{xx}\|_{\infty}};\frac 1{\|f_{yy}\|_{\infty}}\right\}}.
\end{equation}
\end{lemma}

%\begin{proof} Let $\bar{n}=(u,v)$ be an arbitrary unit vector. Then for an arbitrary point $(x,y)\in D$
%$$
%  \frac{\partial^2 f}{\partial \bar{n}^2}(x,y)=f_{xx}(x,y)u^2+2f_{xy}(x,y)uv+f_{yy}(x,y)v^2.
%$$
%Note that functions $f_{xx}$ and $f_{yy}$ have the same sign on $D$.
%Without loss of generality we may assume that $f_{xx}(x,y)>0$ for
%all $(x,y)\in D$. Since $u^2+v^2=1$ then either $u^2\ge\frac 12$, or
%$v^2\ge\frac 12$. If $u^2>\frac 12$ then
%$$
%  \begin{array}{rcl}
%  \displaystyle\frac{\partial^2 f}{\partial \bar{n}^2}(x,y) & = & \displaystyle\left(\frac{f_{xy}^2(x,y)}{f_{yy}(x,y)}u^2 + 2f_{xy}(x,y)uv+f_{yy}(x,y)v^2\right) + %\left(f_{xx}(x,y) - \frac{f_{xy}^2(x,y)}{f_{yy}(x,y)}\right)u^2 \\[10pt]
%  & \ge & \displaystyle\left(f_{xx}(x,y) - \frac{f_{xy}^2(x,y)}{f_{yy}(x,y)}\right)u^2 \ge %\frac{K}{2\|f_{yy}\|_{\infty}}.
%  \end{array}
%$$
%Similarly, if $v^2\ge\frac12$
%$$
%  \frac{\partial^2 f}{\partial \bar{n}^2}(x,y)\ge \frac{K}{2\|f_{xx}\|_{\infty}}.
%$$
%The last two inequalities prove ~(\ref{const}). \end{proof}

The following geometric lemma, {together with Lemma~\ref{L9}}, plays crucial role in the proof of the lower estimate in Theorem~\ref{Th2}.  

\begin{lemma}
\label{L10}
Let $T$ be an arbitrary triangle in the plane with $\textrm{diam}\,T\leqslant \sqrt{2}$. Let also $O$ be an arbitrary point inside $T$, $\delta>0$ be a fixed number and $K_{\delta}:= \min\left\{1;\delta^{2}/2\right\}$. Then there exists triangle $T'$ which lies completely in the intersection of $T$ and the disk $B$ centered at $O$ and having the radius $\delta$ such that 
\begin{equation}
\label{square}
	|T'|
	\geqslant 
	K_{\delta}^2\cdot|T| 
		\qquad \textrm{and}\qquad 
	\textrm{diam}\,T' 
	\geqslant 
	K_{\delta}\cdot\textrm{diam}\,T.
\end{equation}
\end{lemma}

The proof of this result requires only elementary geometric observations and is omitted here.

\section{Error of asymmetric approximation of $C^2$ functions by linear splines: estimate from above} 
\label{S7}

Recall that we consider functions  $f: D\to \RR$ with Hessian that is nonnegative on $D$. For definiteness, we assume that function is {\it convex}. In this section we will show that for every such function and $1\leqslant p<\infty$ 
\begin{equation}
\label{roll}
	\limsup_{N\to\infty}N\cdot  R_N(f,L_{p; \alpha , \beta })
	\leqslant 
	2^{-1} C_{p;\alpha,\beta} \left\| \sqrt H\right\|_{\frac{p}{p+1}}.
\end{equation}
Remark that letting $p\to\infty$ we also could prove lower estimate~(\ref{roll}) in case $p=\infty$. To prove~(\ref{roll}) in Section~\ref{S7.2}, for every $\varepsilon>0$ we will construct a suitable family of triangulations $\{\triangle_N^{\varepsilon}\}_{N=1}^{\infty}$ and the family of corresponding piecewise-linear splines $\left\{s_N^{\varepsilon}\right\}_{N=1}^{\infty}$, $s_N^{\varepsilon}\in S\left(\triangle_N^{\varepsilon}\right)$. Following that, in Section~\ref{S7.3}, we will show that 
$$
	\lim\limits_{\varepsilon\to 0}\left(\limsup\limits_{N\to\infty} N\cdot\left\|f-s_N^{\varepsilon}\right\|_{p;\alpha,\beta} \right) 
	\leqslant 
	2^{-1}C_{p;\alpha,\beta}  \left\|\sqrt H\right\|_{\frac{p}{p+1}}.
$$
The latter inequality implies the desired estimate~(\ref{roll}).

In what follows, we fix the number $1\leqslant p<\infty$.

\subsection{Additional notations} 
\label{S7.1} 

This subsection contains several simple, yet important observations, and notations that will be used throughout this section.   
%$$
%  \begin{array}{rcl} 
%  \lambda_{\min}& := & \displaystyle \frac{1}{2}\left(\frac{f_{xx}+f_{yy}}{2} - \sqrt{\left(\frac{f_{xx}+f_{yy}}{2}\right)^2 - \left(f_{xx}f_{yy} - f_{xy}^2\right)}\right), \\ 
%  \lambda_{\max} & := & \displaystyle  \frac{1}{2}\left(\frac{f_{xx}+f_{yy}}{2} + \sqrt{\left(\frac{f_{xx}+f_{yy}}{2}\right)^2 - \left(f_{xx}f_{yy} - f_{xy}^2\right)}\right), 
%  \end{array}
%$$

%First of all, for $\delta>0$ define the {\it modulus of continuity} of $g \in C(D)$ by 
%$$
%\omega(g, \delta):=\sup \{|g(x,y)-g(x',y')|: \;\; |x-x'|\le\delta;\; |y-y'| \le \delta;\; (x,y), (x',y') \in D\}.
%$$
%For a function $f\in C^2(D)$, we set
%\begin{equation} \label{m1}
  %\omega_2(f,\delta):=\max\{\omega(f_{xx},\delta),\, \omega(f_{yy},\delta),\, \omega(f_{xy},\delta)\}.
%\end{equation}

\begin{lemma}
\label{L3}
Let $f \in C^2(D)$. If $P_2=P_2(f;x,y;x_0,y_0)$ denotes the second degree Taylor polynomial of $f$ at a point $(x_0,y_0)$ inside the square $D_h \subset D$ with side length equal to $h$, then we have the following estimate:
$$
	\|f-P_2\|_{L_{\infty}(D_h)}
	\leqslant 
	2h^2\,\omega_2(f,h).
$$
\end{lemma}
%where $\omega_2(f,\delta)$ is defined in~(\ref{m1}).}

This lemma is obvious, and we omit the proof here. 

Now we consider functions
$$
	\lambda_{\min} (x,y)
	:= 
	\left(f_{xx} + f_{yy}\right)/4 - \sqrt{\left(f_{xx}+f_{yy}\right)^2/16 - \left(f_{xx}f_{yy}-f^2_{xy}\right)/4},
$$
$$
	\lambda_{\max} (x,y)
	:= 
	\left(f_{xx} + f_{yy}\right)/4 + \sqrt{\left(f_{xx}+f_{yy}\right)^2/16 - \left(f_{xx}f_{yy}-f^2_{xy}\right)/4}.
$$
For every $\varepsilon>0$, we define the following set
$$
	G_{0;2\varepsilon}
	:=
	\left\{(x,y)\in D\;:\;0<\lambda_{\min}(x,y)<\varepsilon\right\}. 
$$
Note that $\mu\left(G_{0;2\varepsilon}\right)\to 0$ as $\varepsilon\to 0$. 

In addition, for a fixed $\varepsilon \in (0,1)$ and every $N \in \NN$, we define
\begin{equation}
\label{m_N}
	m_N
	=
	m_N(\varepsilon)
	:=
	\min\left\{	m>0: 
		\;\; 2m^{-2}\max{\{\alpha,\beta\}}\, \omega_2\left(f,m^{-1}\right) 
		\leqslant 
		\varepsilon N^{-1}
	\right\}.
\end{equation}

Observe that $m_N \to \infty$ as $N \to \infty$. In addition, note that
\begin{equation}
\label{star}
	m_N^2N^{-1} \to \infty, \qquad N \to \infty,
\end{equation}
i.e. $m_N=o\left(\sqrt{N}\right)$ as $N\to \infty$ (see, for instance, Section~4 in~\cite{us}). In what follows we will also assume that $N$ is large enough so that inequality $\omega\left(\lambda_{\min},m_N^{-1}\right)\leqslant \varepsilon$ holds true. 

Now, let us subdivide the square $D$ into squares of size $\displaystyle m_N^{-1}\times m_N^{-1}$ with the sides parallel to the sides of $D$. By $D_i^N:=D_i^N(\varepsilon)$, $i=1,\dots,m_N^2$, we denote the resulting squares enumerated in an arbitrary order. 

For $i\in\left\{1,\ldots,m_N^2\right\}$, the center of the square $D_i^N$ we denote by $\left(x_i^N,y_i^N\right)$. Let
$$
	A_i^N
	:=
	f_{xx}\left(x_i^N, y_i^N\right)/2, \;\; 
	B_i^N
	:= 
	f_{yy}\left(x_i^N,y_i^N\right)/2, \;\;
	C_i^N
	:=
	f_{xy}\left(x_i^N, y_i^N\right)/2, 
$$
$$
	\lambda_{\min;i}^N
	:=
	\left(A_i^N+B_i^N\right)/2-\sqrt{\left(A_i^N+B_i^N\right)^2/4-\left(A^N_iB^N_i-(C^N_i)^2\right)},
$$
$$
	\lambda_{\max;i}^N
	:=
	\left(A_i^N+B_i^N\right)/2+\sqrt{\left(A_i^N+B_i^N\right)^2/4-\left(A^N_iB^N_i-(C^N_i)^2\right)}. 
$$
Note that  
$$
	H\left(f;x_i^N,y_i^N\right)
	=
	4\left(A^N_iB_i^N-\left(C_i^N\right)^2\right) 
	= 
	4\lambda_{\min;i}^N\lambda_{\max;i}^N.
$$
Finally, let $f_{N;i}$ be the second degree Taylor polynomial of $f$ constructed at the point $\left(x_i^N,y_i^N\right)$ and denote its quadratic part by
$$
	Q_{N,i} 
	= 
	Q_{N,i}(x,y)
	:=
	A_i^N x^2+2C_i^N xy + B_i^N y^2,\qquad (x,y)\in\RR^2.
$$
%and
%$$
%  \omega_{\lambda}(\delta):=\omega\left(\lambda_{\min},\delta\right),\qquad \delta>0. 
%$$}

\subsection{Construction of the family of ``good'' triangulations and the family of corresponding splines} 
\label{S7.2}

In what follows, let the fixed number $\varepsilon>0$ satisfy the restriction  $23 \varepsilon + 2\mu\left(G_{0;2\varepsilon}\right)<1$. Our aim in the present subsection is to construct ``nearly'' optimal triangulation $\triangle_{N;i}=\triangle_{N;i}(\varepsilon)$ of the square $D_i^N$ for every $i=1,\ldots,m_N^2$. Then we ``glue'' these triangulations into one triangulation $\triangle_N^{\varepsilon}$ of the square $D$. Once the construction of triangulation $\triangle_N^{\varepsilon}$ is complete, we will define the spline $s_N^{\varepsilon}\in \mathcal{S}\left(\triangle_N^{\varepsilon}\right)$ which is ``nearly'' optimal. 

We will split all the set of indices $i=1,...,m_N^2$ into the following four groups: 
$$
	\begin{array}{rcl}
		I_1(\varepsilon;N) 
			& 
		:= 
			&
		\left\{
			i\in\left\{1,\ldots,m_N^2\right\}\;:\;
				\lambda_{\min;i}^N
				\geqslant
				\varepsilon
		\right\}, 
		\\ [5pt]
		I_2(\varepsilon;N) 
			& 
		:= 
			&
		\left\{
			i\in\left\{1,\ldots,m_N^2\right\}\;:\;
				\omega\left(\lambda_{\min},m_N^{-1}\right)
				< 
				\lambda_{\min;i}^N
				<
				\varepsilon
		\right\},
		\\ [5pt]
		I_3(\varepsilon;N) 
			& 
		:= 
			& 
		\left\{
			i\in\left\{1,\ldots,m_N^2\right\}\;:\;
				\lambda_{\min;i}^N 
				\leqslant 
				\omega\left(\lambda_{\min},m_N^{-1}\right),\; 
				\lambda_{\max;i}^N
				\geqslant
				\varepsilon^2
		\right\},
		\\[5pt]
		I_4(\varepsilon;N) 
			& 
		:= 
			& 
		\left\{
			i\in\left\{1,\ldots,m_N^2\right\}\;:\;
				\lambda_{\min;i}^N
				< 
				\omega\left(\lambda_{\min},m_N^{-1}\right),\;
				\lambda_{\max;i}^N
				<
				\varepsilon^2
		\right\}. 
	\end{array}
$$
Evidently, the sets $I_1(\varepsilon;N)$, $I_2(\varepsilon;N)$, $I_3(\varepsilon;N)$ and $I_4(\varepsilon;N)$ are pairwise non-intersecting for every $N$. 
%Let us construct triangulations $\triangle_{N;i}$ independently in four cases: $\left.a\right)$ $i\in I_1(\varepsilon;N)$, $\left.b\right)$ $i\in I_2(\varepsilon;N)$, $\left.c\right)$ $i\in I_3(\varepsilon;N)$ and $\left.d\right)$ $i\in I_4(\varepsilon;N)$. 

Let us describe how we will construct triangulations of each $D_i^N$ depending on which out of four above groups index $i$ is in.

$\bf\left.a\right) $ For this $i\in I_1(\varepsilon;N)$, we set $
	n_i^N
	=
	n_i^N(\varepsilon)
	:=
	\left[
		\frac{\displaystyle N(1-23\varepsilon-2\mu\left(G_{0;2\varepsilon}\right)) H^{\frac{p}{2(p+1)}}\left(f;x_i^N,y_i^N\right)}{\displaystyle \sum_{j\in I_1(\varepsilon;N)} H^{\frac{p}{2(p+1)}}\left(f;x_j^N, y_j^N\right)}
	\right] + 1.
$
Here $[a]$ stands for the integer part of a real number $a$. This quantity $n_i^N$ is the number of triangles in the triangulation of $D^N_i$ (before refining shapes along the boundary) and is obtained by minimizing the global error under the condition $\displaystyle \sum_{i\in I_1(\varepsilon; N)} n^N_i \approx N$.

Since  
$$
  n_i^N > (1-23\varepsilon-2\mu\left(G_{0;2\varepsilon}\right)) (4\varepsilon^2)^{\frac{p}{2(p+1)}}  \|H\|_{\infty}^{-\frac{p}{2(p+1)}}\cdot m_N^{-2}N,
$$
and due to relation~(\ref{star}), we see that $n_i^N$ tends to infinity as $N$ gets large.

In order to formalize further constructions, for an arbitrary triangle $T$ in the plane, by $\textrm{Til}(T)$ we denote the tiling of the plane, generated by $T$ in the following way: we take a triangle $\widetilde{T}$ which is symmetric to $T$ with respect to the midpoint of one of its sides, and then we tile $\RR^2$ with the shifts of $T\cup\widetilde{T}$. 

Let us describe the algorithm for construction of triangulation $\triangle_{N;i}$: 

\begin{enumerate}[topsep=0pt, partopsep=0pt, itemsep=2pt, parsep=2pt]
%\item Let $F_{1}$ and $F_{2}$ be transformations of the form~(\ref{H1}) and~(\ref{H2}) respectively, corresponding to the quadratic function ${Q_{N,i}}$, and set $F:=F_1\circ F_2$.
%\item Let $T$ be an arbitrary triangle which solves Problem~1 (recall that $T$ is equilateral), and consider the triangle $F(T)$.
\item Let $T_i^N$ be the triangle that delivers infimum in~(\ref{Pr1}).
\item By $T_i^N$ we denote a re-scaling of $F(T)$ such that $ \left|T_i^N\right|= m_N^{-2} \left(n_i^N\right)^{-1}$.
\item With the help of the triangle $T_i^N$ we generate the tiling $\textrm{Til}(T_i^N)$ of the plane. 
\item Every triangle from $\textrm{Til}(T_i^N)$ that lies completely inside the square $D_i^N$ we include into triangulation $\triangle_{N;i}$. 
\item For every triangle $T\in \textrm{Til}(T_i^N)$ that has common points with the boundary of $D_i^N$, we consider the intersection $T\cap D_i^N$. Evidently, it is a polygon with at most seven vertices. We split this polygon into at most five triangles without adding new vertices and include them into triangulation $\triangle_{N;i}$.
\end{enumerate}

Let us estimate the number of triangles in $\triangle_{N;i}$. Since the quadratic form ${Q_{N,i}}$ satisfies the conditions of Lemma~\ref{L8}, we derive that there exists a constant $c_1=c_1(\varepsilon)$, independent of $N$, such that 
\begin{equation}
\label{trtr}
	\textrm{diam}\,T_i^N 
	\leqslant 
	c_1 m_N^{-1}\left(n_i^N\right)^{-1/2}, \qquad N\to\infty.
\end{equation}
Consequently, the number of triangles $T\in \textrm{Til}\left(T_i^N\right)$ that have nonempty intersection with the boundary of $D_i^N$ is $O\left(\sqrt{n_i^N}\right)$ as $N\to\infty$. Therefore, the total number of triangles in ${\triangle}_{N;i}$ is
$$
	n_i^N + O\left(\sqrt{n_i^N}\right) 
	= 
	n_i^N + o\left(m_N^{-2}N\right),\qquad N\to\infty. 
$$
%Remark, that the spline $s_N(\varepsilon)$ is continuous. Indeed, note firstly that if $T=\triangle ABC$ and $T'=\triangle A'B'C'$ are two equal equilateral triangles, then $d(q,T,L_{p;\alpha,\beta}(T)) = d(q,T',L_{p;\alpha,\beta}(T')) $ and, moreover, the functions $q-l(q,T;\alpha,\beta)$ and $q-l(q,T';\alpha,\beta)$ attain the same value at the vertices of $T$ and $T'$ respectively. Consequently, the continuity of $s_N(\varepsilon)$ on the union of all triangles from $G(T_i^N)$ that lie in the interior of $D_i^N$ is guaranted automatically. The continuity of $s_N(\varepsilon)$ on the union of triangles from $\triangle_N(\varepsilon)$ that have non-empty intersection with $\bigcup\limits_{i=1}^{m_N^2}\partial D_i^N$ follows from the way of extending this spline. Therefore, the spline $s_N(\varepsilon)$ is continuous on $D$. 

$\bf \left.b\right)$ Let $i\in I_2(\varepsilon;N)$. Since $m_N^{-2}N\to \infty$ with $N$, for each $N$ there exists an integer, denote it by $r_1=r_1(N)$, such that
$$
	2^{-1}m_N^{-2} N 
	\leqslant 
	r_1^2 
	\leqslant 
	m_N^{-2}N. 
$$
Let us subdivide the square $D_i^N$ into squares of the size $\displaystyle m_N^{-1} r_1^{-1}\times m_N^{-1} r_1^{-1}$ whose sides are parallel to the sides of $D_i^N$. Then inside each small square we draw one of its diagonals. Thus, we obtain the triangulation $\triangle_{N;i}$ of $D_i^N$ consisting of $2r_1^2$ equal isosceles right triangles. In addition, select an arbitrary such triangle and denote it by $T_i^N$. 

$\bf\left.c\right)$ Let $i\in I_3(\varepsilon;N)$. Denote by $\overline{\xi}_i$ and $\overline{\eta}_i$ the eigenvectors of the quadratic form {$Q_{N,i}$ corresponding to the eigenvalues $\lambda_{\min;i}^N$ and $\lambda_{\max;i}^N$}. For this $\varepsilon>0$, and for each $N$, there exists an integer, denote it by $r_2$, such that
$$
	2^{-1}\varepsilon m_N^{-2}N 
	\leqslant 
	r_2^2 
	\leqslant 
	\varepsilon m_N^{-2}N. 
$$
Let $\Pi_i^N$ be the rectangle of the size $\displaystyle m_N^{-1}\times m_N^{-1} r_2^{-2}$ whose sides are parallel to vectors $\overline{\xi}_i$ and $\overline{\eta}_i$, {respectively}. We will draw inside this rectangle one of its diagonals and denote any of two constructed triangles by $T_i^N$. Let $T$ be an arbitrary triangle from $\textrm{Til}\left(T_i^N\right)$. If $T$ lies completely in the interior of $D_i^N$ then we include it into triangulation $\triangle_{N;i}$. Otherwise, we split every polygon, which is intersection of $D_i^N\cap T$, into at most five triangles without adding new vertices and include them into triangulation $\triangle_{N;i}$. Note that the number of triangles in $\triangle_{N;i}$ does not exceed $10r_2^2$. 

$\bf \left.d\right)$ Finally, let $i\in I_4(\varepsilon;N)$. Similarly to the case $\left.b\right)$, we subdivide the square $D_i^N$ into squares of the size $\displaystyle m_N^{-1}r_2^{-1}\times m_N^{-1}r_2^{-1}$ whose sides are parallel to the sides of $D_i^N$. Then inside each small square we draw one of its diagonals. Thus, we obtained the triangulation $\triangle_{N;i}$ of $D_i^N$ consisting of $2r_2^2$ equal isosceles right triangles (denoted by $T_i^N$). 

Let us estimate the overall number $\widetilde{N}$ of triangles in above-constructed triangulations $\triangle_{N;i}$. Due to the choice of numbers $r_1$ and $r_2$ we obtain  
$$
	\begin{array}{rcl}
		\widetilde{N} 
			& 
		\leqslant 
			& 
		\displaystyle \sum\limits_{i\in I_1(\varepsilon;N)} \left(n_i^N + o\left(m_N^{-2}N\right)\right) + \sum\limits_{i\in I_2(\varepsilon;N)} 2r_1^2 + \sum\limits_{i\in I_3(\varepsilon;N)} 10 r_2^2 + \sum\limits_{i\in I_4(\varepsilon;N)} 2r_2^2 
		\\ [10pt]
			& 
		\leqslant 
			& 
		\displaystyle (1-23\varepsilon - 2\mu\left(G_{0;2\varepsilon}\right))N  + o(N) + 2\mu\left(G_{0;2\varepsilon}\right) N + 12m_N^2r_2^2
		\\[5pt]
			& 
		\leqslant 
			& 
		\displaystyle \left(1-23\varepsilon\right)N + o\left(N\right)  + 12\varepsilon N 
		= 
		\left(1-11\varepsilon\right)N + o(N) \qquad\textrm{as}\qquad N\to\infty.
	\end{array}
$$

Now we ``glue'' triangulations $\triangle_{N;i}$ according to the following rule: 
\begin{enumerate}[topsep=0pt, partopsep=0pt, itemsep=2pt, parsep=2pt]
\item We include into triangulation $\triangle_N^{\varepsilon}$ every triangle $T\in \triangle_{N;i}$, $i=1,\ldots,m_N^2$, which does not intersect with the boundary of the square $D_i^N$.
\item For every $i=1,\dots,m_N^2$, denote by $W_i^N$ the set of the vertices of triangulation $\triangle_{N;i}$ which lie on the boundary of $D_i^N$. For arbitrary $i,j=1,\ldots,m_N^2$, $i\ne j$, we set $S_{i,j}=D_i^N\cap D_j^N$.
\item We subdivide every triangle $T\in\triangle_{N;i}$ that has non-empty intersection with $S_{i,j}$  by joining the vertices of $T$ with the points from $W_j^N\cap T$. Finally, we include all obtained triangles into triangulation $\triangle_N^{\varepsilon}$.
\end{enumerate} 

Let us estimate the number $\widehat{N}$ of triangles in $\triangle_N^{\varepsilon}$. By $\# A$ we denote the number of points in a finite set $A$. Then 
$$
	\widehat{N} 
	\leqslant 
	\widetilde{N} + \sum\limits_{i=1}^{m_N^2} \#\left(W_i^N\right). 
$$
Note that $\#\left(W_i^N\right)\leqslant 10r_2^2$ for every $i\in I_3(\varepsilon;N)$ and $\#\left(W_i^N\right) = o\left(m_N^{-2}N\right)$ otherwise. Hence, 
$$
	\widehat{N} 
	\leqslant 
	(1-11\varepsilon)N + o(N) + m_N^2\cdot\left(10 r_2^2 + o\left(m_N^{-2}N\right)\right) 
	\leqslant 
	(1-\varepsilon)N + o\left(N\right). 
$$ 
Therefore, $\widehat{N}\leqslant N$ for all $N$ large enough.

Now we are ready to construct the ``nearly'' optimal spline $s_N^{\varepsilon}$ on the triangulation $\triangle_N^{\varepsilon}$. For $i\in I_1(\varepsilon;N)$ and $T\in \Delta_N^{\varepsilon}\cap \textrm{int}\,D_i^N$, let $s_N^{\varepsilon}$ be the sum of two linear polynomials:  $f_{N,i} - Q_{N,i}$ and polynomial of the best approximation of $f_{N,i}$ on triangle $T$. Then due to Remark~4 we see that spline $s_N^{\varepsilon}$ is continuous on the union of interior triangles from $\Delta_N^{\varepsilon}\cap \textrm{int}\,D_i^N$. 

Finally, let $s_N^{\varepsilon}$ interpolate the function $f$ at the remaining vertices of $\triangle_N^{\varepsilon}$, i.e. at the vertices located in the interior of squares $D_i^N$ with $i\in I_2(\varepsilon;N)\cup I_3(\varepsilon;N)\cup I_4(\varepsilon;N)$ as well as at the vertices located along the boundaries of all $D_i^N$'s. This would automatically ``glue'' the spline $s_N^{\varepsilon}$.

%In order to ``glue'' the spline $s_N(\varepsilon)$ we extend the already constructed spline in such a way that it interpolates the function $f$  at the vertices of the triangulation $\triangle_N(\varepsilon)$, at which it has not been defined yet. The continuity of the spline $s_N(\varepsilon)$ is guaranteed, since the spline is uniquely determined by the values at the vertices of triangulation $\triangle_N(\varepsilon)$. 

Therefore, for every sufficiently small $\varepsilon>0$, there exists $N(\varepsilon)\in\NN$ such that for each  $N>N(\varepsilon)$ we have constructed the triangulation $\Delta_N^{\varepsilon}$ with at most $N$ triangles and corresponding continuous piecewise linear spline $s_N^{\varepsilon}\in\mathcal{S}\left(\Delta_N^{\varepsilon}\right)$.

\subsection{The proof of estimate from above} 
\label{S7.3}

In this subsection we will prove that 
\begin{equation}
\label{main}
	\lim\limits_{\varepsilon\to 0} \left(\liminf\limits_{N\to\infty} N\cdot \left\|f-s_N^{\varepsilon}\right\|_{p;\alpha,\beta} \right)
	\leqslant 
	2^{-1} C_{p;\alpha,\beta}\left\|\sqrt{H}\right\|_{\frac{p}{p+1}}. 
\end{equation}
However, we need several preliminary results concerning the estimates for the deviation of spline $s_N^{\varepsilon}$ from the function $f$ on squares $D_i^N$, $i=1,\ldots,m_N^2$. By $\triangle_{i}^N$ and $\widetilde{\triangle}_i^N$, $i=1,\ldots,m_N^2$, we denote the set of triangles $T\in \triangle_N^{\varepsilon}$ that lie in the square $D_i^N$ and in the interior of $D_i^N$, respectively. Let also 
$$
	\Omega_i^N 
	:= 
	\left\{(x,y)\;:\;\;\exists \;\;T \in \widetilde{\triangle}_i^N \;\;\;\hbox{such that}\;\;\; (x,y)\in T\right\}. 
$$

Let us find the upper estimates for the quantity $\left\|f-s_N^{\varepsilon}\right\|_{L_{p;\alpha,\beta}\left(D_i^N\right)}$ separately in two different situations: $\left.1\right)$~$i\in \left\{1,\ldots,m_N^2\right\}\setminus I_3(\varepsilon;N)$ and $\left.2\right)$ $i\in I_3(\varepsilon;N)$. 

$\bf\left.1\right)$ As a first step let $i$ be an arbitrary index from the set $\left\{1,\ldots,m_N^2\right\}\setminus I_3(\varepsilon;N)$. Note that 
\begin{equation}
\label{first_rofl}
	\left\|f-s_N^{\varepsilon}\right\|_{L_{p;\alpha,\beta}\left(D_i^N\right)}^p  
	= 
	\left\|f-s_N^{\varepsilon}\right\|_{L_{p;\alpha,\beta}\left(\Omega_i^N\right)}^p + \sum\limits_{T\in \triangle_i^N\setminus\widetilde{\triangle}_i^N}\left\|f - s_N^{\varepsilon}\right\|_{L_{p;\alpha,\beta}\left(T\right)}^p. 
\end{equation}
Due to the triangle inequality, we have
\begin{equation}
\label{second_rofl}
	\left\|f-s_N^{\varepsilon}\right\|_{L_{p;\alpha,\beta}\left(\Omega_i^N\right)} 
	\leqslant 
	\left\|f-f_{N;i}\right\|_{L_{p;\alpha,\beta}\left(\Omega_i^N\right)} + \left\|f_{N;i}-s_N^{\varepsilon}\right\|_{L_{p;\alpha,\beta}\left(\Omega_i^N\right)}. 
\end{equation}
We estimate the first term in the right-hand side of~(\ref{second_rofl}) with the help of Lemma~\ref{L3}: 
\begin{equation}
\label{third_rofl}
	\begin{array}{rcl}
		\displaystyle \left\|f-f_{N;i}\right\|_{L_{p;\alpha,\beta}\left(\Omega_i^N\right)} 
			& 
		\leqslant 
			& 
		\displaystyle  \max\{\alpha;\beta\} \left\|f-f_{N;i}\right\|_{L_{\infty}\left(D_i^N\right)} \left|D_i^N\right|^{1/p}
		\\ [5pt]
  			& 
  		\leqslant 
  			& 
  		\displaystyle \max\{\alpha;\beta\} m_N^{-2-2/p}\;\omega_2\left(f, m_N^{-1}\right) 
  		\leqslant 
  		\displaystyle \varepsilon m_N^{-2/p}N.  
	\end{array}
\end{equation}

As for the second term in the right-hand side of~(\ref{second_rofl}), we observe that
\begin{equation}
\label{29}
	\left\|f_{N;i}-s_N^{\varepsilon}\right\|_{L_{p;\alpha,\beta}\left(\Omega_i^N\right)}^p
	= 
	\sum\limits_{T\in \widetilde{\triangle}_i^N} \left\|f_{N;i}-s_N^{\varepsilon}\right\|^p_{L_{p;\alpha,\beta}\left(T\right)} 
	= 
	\left(\#\widetilde{\triangle}_i^N\right) \cdot\left\|f_{N;i}-s_N^{\varepsilon}\right\|^p_{L_{p;\alpha,\beta}\left(T_i^N\right)}. 
\end{equation}
%where notation $\#A$ stands for the number of elements in a finite set $A$. 
In order to obtain upper estimates for $\left\|f_{N;i}-s_N^{\varepsilon}\right\|_{L_{p;\alpha,\beta}\left(T_i^N\right)}$ we should consider three cases: $\left.a\right)$ $i\in I_1(\varepsilon;N)$, $\left.b\right)$ $i\in I_2(\varepsilon;N)$, and $\left.c\right)$ $i\in I_4(\varepsilon;N)$. 

$\bf\left.a\right)$ First, we assume that $i\in I_1(\varepsilon;N)$. Due to the algorithm for construction of triangulation $\triangle_N(\varepsilon)$ described in the previous subsection and due to Corollary~\ref{cor}, we have that $\#\widetilde{\triangle}_i^N \leqslant n_i^N$ and 
$$
	\left\|f_{N;i} - s_N^{\varepsilon}\right\|_{L_{p;\alpha,\beta}\left(T_i^N\right)} 
	= 
	2^{-1}C_{p;\alpha,\beta}H^{1/2}\left(f;x_i^N,y_i^N\right)\left|T_i^N\right|^{1+1/p} 
	= 
	\frac{C_{p;\alpha,\beta}H^{1/2}\left(f;x_i^N,y_i^N\right)}{2 m_N^{2+2/p}\left(n_i^N\right)^{1+1/p}}. 
$$
From here and from~(\ref{29}), we obtain 
\begin{equation}
\label{first_Omega}
	\left\|f_{N;i}-s_N^{\varepsilon}\right\|_{L_{p;\alpha,\beta}\left(\Omega_i^N\right)}^p
	\leqslant 
	2^{-p}C^{p}_{p;\alpha,\beta}\cdot H^{p/2}\left(f;x_i^N,y_i^N\right) \cdot m_N^{-2\left(p+1\right)}\left(n_i^N\right)^{-p}. 
\end{equation}

$\bf\left.b\right)$ Next, we assume that $i\in I_2(\varepsilon;N)$. Then $\#\widetilde{\triangle}_i^N \leqslant 2r_1^2$. For every triangle $T$ and continuous on $T$ function $g$, by $s_{g,T}$ we denote the linear function interpolating $g$ at the vertices of $T$. Since $\displaystyle\left|\frac{\partial^2 f_{N;i}}{\partial\overline{n}^2}\right|\leqslant \lambda_{\max;i}^N$ for every unit vector $\overline{n}$, after change of variables, we have (see ~\cite{Sub, Kil}) 
$$
	\begin{array}{rcl}
		\displaystyle \left\|f_{N;i}-s_N^{\varepsilon}\right\|_{L_{p;\alpha,\beta}\left(T_i^N\right)} 
			& 
		\leqslant 
			& 
		\displaystyle  \frac{\max\{\alpha;\beta\} \lambda_{\max;i}^N}{2} \left(\int_{T_i^N}\left|q(x,y) - s_{q,T_i^N}(x,y)\right|^p\,dx dy\right)^{1/p} 
		\\[10pt]
			& 
		\leqslant 
			& 
		\displaystyle 2^{-1} k_1 \left(m_N^2r_1^2\right)^{-1-1/p} 
		\leqslant 
		k_1 \left(m_N^2r_1^2\right)^{-1-1/p}. 
	\end{array}
$$
Here we recall that $q(x,y)=x^2+y^2$, and denote by $k_1,k_2,\ldots$ constants that are independent of $N$ and $\varepsilon$. Therefore, 
\begin{equation}
\label{second_Omega}
	\left\|f_{N;i}-s_N^{\varepsilon}\right\|_{L_{p;\alpha,\beta}\left(\Omega_i^N\right)}^p
	\leqslant 
	2k_1^p m_N^{-2}N^p.
\end{equation}

$\bf \left.c\right)$ Finally, let $i\in I_4(\varepsilon;N)$. Then like to the previous case we obtain that $\#\widetilde{\triangle}_i^N \leqslant 2r_2^2$ and 
%$$
%  \left\|f_{N;i}-s_N(\varepsilon)\right\|_{L_{p;\alpha,\beta}\left(T_i^N\right)} \le \frac{k_1\varepsilon}{N \left(m_N^2 r_2^2\right)^{\frac 1p}}.
%$$
%Hence, 
\begin{equation}
\label{third_Omega}
	\left\|f_{N;i}-s_N^{\varepsilon}\right\|_{L_{p;\alpha,\beta}\left(\Omega_i^N\right)}^p
	\leqslant 
	2k_1^p\varepsilon^p m_N^{-2} N^p.
\end{equation}

The analysis of three above cases is complete. Next we will estimate the deviation of spline $s_N(\varepsilon)$ from the function $f$ on arbitrary triangle $T\in \triangle_i^N\setminus \widetilde{\triangle}_i^N$. For every such triangle, 
$$
	\left\|f-s_N^{\varepsilon}\right\|_{L_{p;\alpha,\beta}(T)} 
	\leqslant 
	\max\{\alpha;\beta\} \left(\left\|f-s_{f,T}\right\|_{L_{p}(T)} + \left\|s_{f,T} - s_N^{\varepsilon}\right\|_{L_p(T)}\right). 
$$
By the triangle inequality we obtain that 
$$
	\left\|f-s_{f,T}\right\|_{L_{p}(T)} 
	\leqslant 
	\left\|f_{N;i}-s_{f_{N;i},T}\right\|_{L_p(T)} + \left\|f-f_{N;i}\right\|_{L_p(T)} + \left\|s_{f,T} - s_{f_{N;i},T}\right\|_{L_p(T)}.
$$
In view of Lemma~\ref{L3} we have that 
$$
	\left\|f-f_{N;i}\right\|_{L_p(T)}
	\leqslant
	\left\|f-f_{N;i}\right\|_{L_{\infty}(T)} |T|^{1/p} 
	\leqslant 
	m_N^{-2}\,\omega_2\left(f,m_N^{-1}\right) |T|^{1/p}
	=
	O\left(N^{-1-1/p}\right), 
$$
$$
	\left\|s_{f,T} - s_{f_{N;i},T}\right\|_{L_p(T)} 
	\leqslant 
	\left\|f-f_{N;i}\right\|_{L_{\infty}(T)} |T|^{1/p} 
	= 
	O\left(N^{-1-1/p}\right),
$$
$$
	\left\|s_{f,T}-s_N^{\varepsilon}\right\|_{L_p(T)} 
	\leqslant 
	\left\|f-f_{N;i}\right\|_{L_{\infty}(T)}|T|^{1/p} 
	= 
	O\left(N^{-1-1/p}\right), 
$$
as $N\to\infty$. In addition, by Corollary~\ref{cor}
$$
	\left\|f_{N;i}-s_{f_{N;i},T}\right\|_{L_p(T)} 
	\leqslant 
	\left\|f_{N;i}-s_{f_{N;i},T_i^N}\right\|_{L_p\left(T_i^N\right)} 
	\leqslant 
	2^{-1} \lambda_{\max;i}^N\cdot \left\|q-s_{q,T_i^N}\right\|_{L_p\left(T_i^N\right)}. 
$$
Due to inequality~(\ref{trtr}) in the case $i\in I_1(\varepsilon;N)$ we have 
$$
	\left\|q-s_{q,T_i^N}\right\|_{L_p\left(T_i^N\right)} 
	\leqslant 
	O\left(N^{-1-1/p}\right). 
$$
In the case $i\in I_2(\varepsilon;N)\cup I_4(\varepsilon;N)$ we have 
$$
	\left\|q-s_{q,T_i^N}\right\|_{L_p\left(T_i^N\right)} 
	= 
	k_2 m_N^{-2-2/p}r_1^{-2-2/p} 
	= 
	O\left(N^{-1-1/p}\right).
$$
Therefore, 
$$
	\left\|f-s_N^{\varepsilon}\right\|_{L_{p;\alpha,\beta}(T)}  
	= 
	O\left(N^{-1- 1/p}\right),\qquad N\to\infty. 
$$

Let us remind that we are considering the case when $i\in\left\{1,\ldots,m_N^2\right\}\setminus I_3(\varepsilon;N)$. Due to the algorithm for construction of triangulation $\triangle_N(\varepsilon)$, we conclude that the number of triangles in $\triangle_i^N\setminus \widetilde{\triangle}_i^N$ is $o\left(m_N^{-2}N\right)$ as $N\to\infty$. This implies that 
\begin{equation}
\label{ostatki}
	\sum\limits_{T\in \triangle_i^N\setminus \widetilde{\triangle}_i^N}\left\|f - s_N^{\varepsilon}\right\|_{L_{p;\alpha,\beta}\left(T\right)}^p 
	= 
	o\left(m_N^{-2} N^{-p}\right). 
\end{equation}

$\bf\left.2\right)$ Now let us consider $i\in I_3(\varepsilon;N)$. It can be easily seen that 
\begin{equation}
\label{eht_28}
	\left\|f- s_N^{\varepsilon}\right\|^p_{L_{p;\alpha,\beta}\left(D_i^N\right)} 
	\leqslant 
	\left(\max\{\alpha,\beta\}\right)^p \sum\limits_{T\in \triangle_i^N} \left\|f-s_N^{\varepsilon}\right\|_{L_p(T)}^p. 
\end{equation}
Let $T$ be an arbitrary triangle in $\triangle_i^N$. Then 
$$
	\left\|f-s_N^{\varepsilon}\right\|_{L_p(T)} 
	\leqslant 
	\left\|f-f_{N;i}\right\|_{L_p(T)} + \left\|f_{N;i} - s_{f_{N;i},T}\right\|_{L_p(T)} + \left\|s_N^{\varepsilon} - s_{f_{N;i},T}\right\|_{L_p(T)}. 
$$
With the help of Lemma~\ref{L3} we obtain the following upper estimates 
$$
	\left\|f-f_{N;i}\right\|_{L_p(T)} 
	\leqslant 
	\left\|f-f_{N;i}\right\|_{L_{\infty}(T)} |T|^{1/p} 
	\leqslant 
	\varepsilon N^{-1}\cdot \left(2m_N^2r_2^2\right)^{-1/p},
$$
$$
	\left\|s_N^{\varepsilon}- s_{f_{N;i},T}\right\|_{L_p(T)} 
	\leqslant 
	\varepsilon N^{-1}\cdot \left(2m_N^2r_2^2\right)^{-1/p}. 
$$

In addition, we set $\tilde{q}(x,y):=\lambda_{\min;i}^N\,x^2+\lambda_{\max;i}^N\,y^2$ and observe that 
$$
	\begin{array}{rcl}
		\displaystyle\left\|f_{N;i} - s_{f_{N;i},T}\right\|_{L_p(T)}^p 
			& 
		\leqslant 
			& 
		\displaystyle\int_0^{m_N^{-1}}\int_0^{m_N^{-1}r_2^{-2}}\left|\tilde{q}(x,y) - s_{\tilde{q},T_i^N}(x,y)\right|^p\,dy\,dx 
		\\ [10pt]
			& 
		\leqslant 
			& 
		\displaystyle \int_0^1\int_0^1 \left|\frac{\omega\left(\lambda_{\min},m_N^{-1}\right)}{m_N^2}u^2 + \frac{\left\|\lambda_{\max}\right\|_{\infty}}{m_N^2r_2^4}v^2-s_{\tilde{q},T_i^N}\left(\frac{u}{m_N},\frac{v}{m_N r_2^2}\right)\right|^p \frac{du\,dv}{m_N^2 r_2^2}
		\\ [10pt]
			& 
		\leqslant 
			&
		\displaystyle\frac{\omega^p\left(\lambda_{\min},m_N^{-1}\right)}{m_N^{2(p+1)}r_2^{2(p+1)}} \cdot\int_0^1\left|u^2-u\right|^p\,du + \frac{\left\|\lambda_{\max}\right\|_{\infty}^p}{m_N^{2(p+1)}r_2^{4p+2}} \cdot\int_0^1 \left|v^2-v\right|^p\,dv. 
	\end{array}
$$
Hence, taking into account (\ref{m_N}), we obtain
\begin{equation}
\label{the_last_inequality}
	\left\|f_{N;i}-s_{f_{N;i},T_i^N}\right\|_{L_p(T)} 
	\leqslant 
	\frac{k_3}{\left(m_N^2 r_2^2\right)^{1/p}}\left(m_N^{-2}\omega\left(\lambda_{\min},m_N^{-1}\right) + \frac{2\left\|\lambda_{\max}\right\|_{\infty}}{\varepsilon N r_2^2}\right)
	\leqslant 
	\frac{k_3(\varepsilon + o(1))}{N\cdot \left(m_N^2 r_2^2\right)^{1/p}} 
\end{equation}
as $N\to\infty$. 

Therefore, we arrive at the following estimate in the case when $i\in I_3(\varepsilon;N)$
\begin{equation}
\label{30.1}
	\|f-s_N^{\varepsilon}\|_{L_p(T)}
	\leqslant 
	\frac{2\varepsilon}{N}\cdot \frac{1}{2^{1/p}\left(m_N^2r^2_2\right)^{1/p}}+ \frac{k_3(\varepsilon + o(1))}{N\cdot \left(m_N^2 r_2^2\right)^{1/p}}
	\leqslant 
	\frac{\left(k_3+2\right)(\varepsilon+o(1))}{N(m_N^2 r_2^2)^{1/p}}.
\end{equation}

Now, we are ready to prove inequality~(\ref{main}). Indeed, 
\begin{equation}
\label{main_inequality}
	\left\|f-s_N^{\varepsilon}\right\|_{p;\alpha,\beta}^p 
	= 
	\sum\limits_{i=1}^{m_N^2} \left\|f-s_N^{\varepsilon}\right\|_{L_{p;\alpha,\beta}\left(D_i^N\right)}^p 
	= 
	\sum\limits_{j=1}^{4}\sum\limits_{i\in I_j(\varepsilon;N)} \left\|f-s_N^{\varepsilon}\right\|_{L_{p;\alpha,\beta}\left(D_i^N\right)}^p. 
\end{equation}
Let us estimate each of four terms in~(\ref{main_inequality}) independently. 

$\bf\left.1\right)$ Combining inequalities~(\ref{first_rofl}),~(\ref{second_rofl}),~(\ref{third_rofl}),~(\ref{first_Omega}) and~(\ref{ostatki}) we see that for every $i\in I_1(\varepsilon;N)$, 
$$
	\begin{array}{rcl}
		\displaystyle\left\|f-s_N^{\varepsilon}\right\|_{L_{p;\alpha,\beta}\left(D_i^N\right)}^p 
			& 
		\leqslant 
			&
		\displaystyle\left[\frac{C_{p;\alpha,\beta}}{2}\cdot\frac{H^{1/2}\left(f;x_i^N,y_i^N\right)}{m_N^{2+2/p}n_i^N} + \frac{\varepsilon}{N m_N^{2/p}}\right]^p + o\left(\frac{1}{m_N^2 N^p}\right) 
		\\ [14pt] 
			& 
		\leqslant 
			& 
		\displaystyle \left(\frac{C_{p;\alpha,\beta}}{2}\right)^{p}\cdot\frac{H^{p/2}\left(f;x_i^N,y_i^N\right)} {m_N^{2\left(p+1\right)}\left(n_i^N\right)^p} + \frac{k_4\varepsilon}{N^p m_N^2} + o\left(\frac{1}{N^p m_N^2}\right),\qquad N\to\infty. 
	\end{array}
$$
From the latter inequality, definition of numbers $n_i^N$, and the Riemann integrability of function $\sqrt{H(f;x,y)}$ we obtain 
\begin{equation}
\label{first_term}
	\begin{array}{l}
		\displaystyle  \sum\limits_{i\in I_1(\varepsilon;N)} \left\|f-s_N^{\varepsilon}\right\|_{L_{p;\alpha,\beta}\left(D_i^N\right)}^p 
		\leqslant 
		\displaystyle \frac{C_{p;\alpha,\beta}^p}{2^p m_N^{2(p+1)}} \sum\limits_{i\in I_1(\varepsilon;N)} \frac{H^{p/2}\left(f;x_i^N,y_i^N\right)}{\left(n_i^N\right)^p} + \frac{k_4\varepsilon + o(1)}{N^p}
		\\ [10pt]
		\qquad\qquad\leqslant 
		\displaystyle\frac{C^p_{p;\alpha,\beta}m_N^{-2(p+1)}}{2^p \left(1-23\varepsilon -2\mu\left(G_{0;2\varepsilon}\right)\right)^p N^p} \left( \displaystyle \sum_{j \in I_1(\varepsilon;N)} H^{\frac{p}{2(p+1)}}(x_j^N, y_j^N)\right)^{p+1} + \frac{k_4\varepsilon + o(1)}{N^p} 
		\\[10pt] 
		\qquad\qquad \leqslant 
		\displaystyle \frac{C_{p;\alpha,\beta}^p}{2^p\left(1-23\varepsilon -2\mu\left(G_{0;2\varepsilon}\right)\right)^p N^p}\left ( \left\|\sqrt{H}\right\|_{\frac p{p+1}}^p + o(1) \right)+\frac{k_4\varepsilon}{N^p}
	\end{array}
\end{equation}
as $N\to\infty$. 

$\bf \left.2\right)$ Let $i\in I_2(\varepsilon;N)$. In view of inequalities~(\ref{first_rofl}),~(\ref{second_rofl}),~(\ref{third_rofl}),~(\ref{second_Omega}), and~(\ref{ostatki}) we obtain 
$$
	\left\|f-s_N^{\varepsilon}\right\|_{L_{p;\alpha,\beta}\left(D_i^N\right)}^p 
	\leqslant 
	\left(\left(2^{1/p}k_1+\varepsilon\right)^p + o(1)\right) m_N^{-2} N^{-p},\qquad N\to\infty. 
$$
From this we derive that 
\begin{equation}
\label{second_term}
	\sum\limits_{i\in I_2(\varepsilon;N)} \left\|f-s_N^{\varepsilon}\right\|_{L_{p;\alpha,\beta}\left(D_i^N\right)}^p 
	\leqslant 
	\mu\left(G_{0;2\varepsilon}\right)\cdot \left(\left(2^{1/p}k_1+\varepsilon\right)^p + o(1)\right)  N^{-p},\qquad N\to\infty. 
\end{equation}

$\bf \left.3\right)$ Combining inequalities~(\ref{eht_28}) and~(\ref{the_last_inequality}), we obtain that 
\begin{equation}
\label{third_term}
	\begin{array}{rcl}
		\displaystyle  \sum\limits_{i\in I_3(\varepsilon;N)} \left\|f-s_N^{\varepsilon}\right\|_{L_{p;\alpha,\beta}\left(D_i^N\right)}^p 
			& 
		\leqslant 
			& 
		\displaystyle \left(\max\{\alpha,\beta\}\right)^p \sum\limits_{i\in I_3(\varepsilon;N)} \sum\limits_{T\in \triangle_i^N} \left\|f-s_N^{\varepsilon}\right\|_{L_p(T)}^p 
		\\[10pt]
			& 
		\leqslant 
			& 
		\displaystyle \left(\max\{\alpha,\beta\}\right)^p \sum\limits_{i\in I_3(\varepsilon;N)} \sum\limits_{T\in \triangle_i^N} \frac{\left(k_3+2\right)^p(\varepsilon + o(1))^p}{N^p m_N^2r_2^2}
		\\[10pt]
			& 
		= 
			& 
		\displaystyle \left(\max\{\alpha,\beta\}\right)^p \cdot 10r_2^2 m_N^2 \frac{\left(k_3+2\right)^p(\varepsilon + o(1))^p}{N^p m_N^2r_2^2} 
		\\ [10pt] 
			& 
		= 
			&
		\displaystyle 10 \frac{\left(\max\{\alpha,\beta\}\left(k_3+2\right)\varepsilon\right)^p + o(1)}{N^p} \quad\textrm{as}\quad N\to\infty.
	\end{array} 
\end{equation}

$\bf\left.4\right)$ Finally, in view of inequalities~(\ref{first_rofl}),~(\ref{second_rofl}),~(\ref{third_rofl}),~(\ref{third_Omega}), and~(\ref{ostatki}) we conclude that 
\begin{equation}
\label{fourth_term}
	\sum\limits_{i\in I_4(\varepsilon;N)} \left\|f-s_N^{\varepsilon}\right\|_{L_{p;\alpha,\beta}\left(D_i^N\right)}^p 
	\leqslant 
	\left((2k_1^p+1)\varepsilon^p + o(1)\right) N^{-p},\qquad N\to\infty. 
\end{equation}

Now, we combine estimates~(\ref{main_inequality}),~(\ref{first_term}),~(\ref{second_term}),~(\ref{third_term}), and~(\ref{fourth_term}). As a result we obtain  
\begin{equation}
\label{new_ineq}
	\begin{array}{rcl}
		\displaystyle   \left\|f-s_N^{\varepsilon}\right\|_{p;\alpha,\beta}^p 
			& 
		\leqslant 
			&
		\displaystyle \frac{C_{p;\alpha,\beta}^p}{2^p\left(1-23\varepsilon -2\mu\left(G_{0;2\varepsilon}\right)\right)^p N^p} \left\|\sqrt{H}\right\|_{\frac p{p+1}}^p + \frac{k_4\varepsilon + \mu\left(G_{0;2\varepsilon}\right)\cdot(2k_1+\varepsilon)^p}{N^p} 
		\\ [10pt] 
			& 
			& 
		\displaystyle \quad + \frac{10\left(\max\{\alpha;\beta\}\right)^p\left(k_3+2\right)^p\varepsilon^p + (2k_1+1)^p\varepsilon^p + o(1)}{N^p},\qquad N\to\infty. 
	\end{array}
\end{equation}
Therefore, 
$$
	\begin{array}{rcl}
		\displaystyle   \limsup\limits_{N\to\infty} N^p\cdot \left\|f-s_N^{\varepsilon}\right\|^p_{p;\alpha,\beta} 
			& 
		\leqslant 
			& 
		\displaystyle \frac{C_{p;\alpha,\beta}}{2\left(1-23\varepsilon -2\mu\left(G_{0;2\varepsilon}\right)\right)} \left\|\sqrt{H}\right\|_{\frac p{p+1}} + k_4\varepsilon + \mu\left(G_{0;2\varepsilon}\right)\cdot(2k_1+\varepsilon)^p 
		\\ [10pt]
			& 
			& 
			\qquad + \displaystyle 10\left(\max\{\alpha;\beta\}\right)^p\left(k_3+2\right)^p\varepsilon^p + (2k_1+1)^p\varepsilon^p. 
	\end{array}
$$
The latter upper estimate implies inequality~(\ref{main_inequality}).

\section{Error of asymmetric approximation of $C^2$ functions by linear splines: lower estimate} \label{S8}

To prove the lower estimate of the optimal error, we need the following lemma. We omit the proof here as the lemma itself is rather evident.

\begin{lemma}
\label{L}
 Let $T$ be an arbitrary triangle. Then for every function $f\in C^2(T)$, $H(f;x,y)\geqslant K>0$ on $T$, it follows that there exists a constant $\Upsilon_f>0$ (independent of $T$) such that 
$$
	E(f,P_1)_{L_{p;\alpha,\beta}(T)}\geqslant K \Upsilon_f(\textrm{diam}\,T)^2|T|^{1/p}.
$$
\end{lemma}

Let the number $1\leqslant p<\infty$ be fixed. In this section we develop ideas of the paper~\cite{BBS} to prove that for every function $f:D\to \RR$ with nonnegative Hessian the following inequality holds true
\begin{equation}
\label{lower_estimate}
	\liminf_{N\to\infty} N \cdot R_N(f,L_{p; \alpha , \beta }) 
	\geqslant 
	2^{-1}C_{p;\alpha,\beta} \left\|\sqrt{H}\right\|_{\frac{p}{p+1}}. 
\end{equation}

For every $\varepsilon>0$, we define the sets $A_{\varepsilon}$ and $F_{\varepsilon}$ in the following way
$$
	\begin{array}{c}
		A_{\varepsilon} 
		:= 
		\left\{(x,y)\in D\;:\;H(f;x,y)<\varepsilon\right\}, 
		\\ [10pt]
		F_{\varepsilon} 
		:= 
		D\setminus A_{\varepsilon} 
		= 
		\left\{(x,y)\in D\;:\;H(f;x,y)\geqslant\varepsilon\right\}. 
	\end{array}
$$
For an arbitrary triangle $T$ in the plane, denote by $\textrm{diam}\,T$ and $|T|$ the length of the longest side and the area of $T$, respectively. In addition, let $U_T$ be an arbitrary point inside $T$. 

Let $N\in\NN$ and let $\triangle=\left\{T_i\right\}_{i=1}^N$ be an arbitrary triangulation of the square $D$ consisting of $N$ triangles. We need to distinguish (in triangulation $\triangle$) several types of triangles: normal, narrow, extra-long, and the triangles where the Hessian of function $f$ is relatively small. To this end for every $N\in\NN$, we set $I_N:=\left\{1,\ldots,N\right \}$, and define the following five subsets of $I_N$: 
\begin{itemize}
\item $M_1(\triangle;\varepsilon):=\{i\in I_N\;:\;T_i\subset A_{2\varepsilon}\}$; 
\item $\displaystyle M_2(\triangle;\varepsilon):=\left\{i\in I_N\;:\; T_i\subset F_{\varepsilon},\; \frac{\left(\textrm{diam}\,T_i\right)^2\omega_{2}(f,\textrm{diam}\,T_i)}{\sqrt{H(f;U_{T_i})}|T_i|} \leqslant \frac{\varepsilon C_{p;\alpha,\beta}}{4\max\{\alpha;\beta\}} \right\}$;
\item $\displaystyle M_3(\triangle;\varepsilon):=\left\{i\in I_N\;:\; T_i\subset F_{\varepsilon},\; \frac{\left(\textrm{diam}\,T_i\right)^2\omega_{2}(f,\textrm{diam}\,T_i)}{\sqrt{H(f;U_{T_i})}|T_i|} > \frac{\varepsilon C_{p;\alpha,\beta}}{4\max\{\alpha;\beta\}},\; \textrm{diam}\,T_i \leqslant \frac{\varepsilon}{\sqrt[4]{N}}\right\}$;
\item $\displaystyle M_4(\triangle;\varepsilon):=\left\{i\in I_N\;:\;T_i\subset F_{\varepsilon},\; \textrm{diam}\,T_i > \varepsilon N^{-1/4}\right\}$; 
\item $M_5(\triangle;\varepsilon) := \left\{i\in I_N\;:\;T_i\cap A_{\varepsilon}\ne\emptyset,\;T_i\cap F_{2\varepsilon}\ne\emptyset\right\}$.
\end{itemize}
According to the given definition, every index $i\in I_N$ belongs to at least one (possibly more) of the sets $M_{j}(\triangle;\varepsilon)$, $j=1,\ldots,5$.

Note that the set $M_3(\triangle;\varepsilon)$ consists of narrow triangles while the sets $M_4(\triangle;\varepsilon)$ and $M_5(\triangle;\varepsilon)$ consists of extra long triangles. In the next three propositions we will show that the overall area of these ``bad'' triangles in ``nearly'' optimal triangulation $\triangle$ is relatively small. 

%Let $\omega_H$ be the modulus of continuity of the Hessian of function $f$, i. e. 
%$$
%  \omega_H(\delta):=\sup\left\{|H(f;x,y) - H(f;x',y')|\;:\;|x-x'| + |y-y'| \le\delta\right\},\qquad \delta>0. 
%$$

\begin{lemma}
\label{L11}
Let $\varepsilon>0$ and let $\left\{\triangle_N\right\}_{N=1}^{\infty}$ be the sequence of triangulations $\triangle_N = \left\{T_i^N\right\}_{i=1}^{N}$ of $D$. If 
\begin{equation}
\label{assumption_of_lemma_X3}
	\liminf\limits_{N\to\infty} \left ( N\cdot\inf\limits_{s\in\mathcal{S}(\triangle_N)}\|f-s\|_{p;\alpha,\beta}\right )<\infty 
\end{equation}
then for all $N$ large enough 
\begin{equation}
\label{lower_area_bound}
	\sum_{i\in M_3(\triangle_N;\varepsilon)} \left|T_i^N\right| < \varepsilon.
\end{equation}
\end{lemma}

\begin{proof} 
Let number $\varepsilon>0$ be fixed. %By $\delta_0$ we denote the minimal positive number such that $\omega_H(\triangle_0)=\varepsilon$. 
Assume to the contrary that there exists a subsequence $\{N_k\}_{k=1}^{\infty}$ of positive integers, such that $N_k\to \infty$ as $k\to \infty$, and
\begin{equation}
\label{1assump}
	\sum_{i\in M_3(\triangle_{N_k};\varepsilon)}\left|T^{N_k}_i\right|
	\geqslant
	\varepsilon.
\end{equation}
Without loss of generality we let $N_k=k$ for every $k\in\NN$. Applying Lemma~\ref{L} and the definition of the set $M_3(\triangle_N;\varepsilon)$, for every $N\in\NN$ and $i\in M_3(\triangle_N;\varepsilon)$, we obtain 
$$
	\begin{array}{rcl}
		E(f,\mathcal{P}_1)_{L_{p;\alpha,\beta}\left(T_i^N\right)} 
			& 
		\geqslant 
			& 
		\displaystyle \varepsilon\Upsilon_{f} \left(\textrm{diam}\,T_i^N\right)^2\left|T_i^N\right|^{1/p}
		>
		\varepsilon\Upsilon_f\cdot\frac{\varepsilon C_{p;\alpha,\beta}}{4\max\{\alpha;\beta\}}\cdot\frac{H^{1/2}\left(f;U_{T_i^N}\right)\left|T_i^N\right|^{1+1/p}}{\omega_2\left(f,{\rm diam}\,T_i^N\right)} 
		\\ [7pt] 
			& 
		=: 
			& 
		\displaystyle \frac{c_2 H^{1/2}\left(f;U_{T_i^N}\right)\left|T_i^N\right|^{1+1/p}}{\omega_2\left(f, \textrm{diam}\,T_i^N\right)}. 
	\end{array}
$$
Here $c_2=c_2(\varepsilon)$ is the constant independent of $N$. Then
$$
	\begin{array}{rcl} 
		\mathcal{F}(\triangle_N) 
			& 
		:= 
			& 
		\displaystyle\inf\limits_{s\in\mathcal{S}(\triangle_N)} \|f-s\|_{p;\alpha,\beta}^p
		\geqslant 
		\sum\limits_{i=1}^N E^p\left(f,\mathcal{P}_1\right)_{L_{p;\alpha,\beta}\left(T_i^N\right)} 
		\geqslant 
		\displaystyle \sum\limits_{i\in M_3(\triangle_N;\varepsilon)} E^p(f,\mathcal{P}_1)_{L_{p;\alpha,\beta}\left(T_i^N\right)} 
		\\
			& 
		> 
			&
		\displaystyle c^p_2 \sum\limits_{i\in M_3(\triangle_N;\varepsilon)} \frac{H^{p/2}\left(f;U_{T_i^N}\right)\left|T_i^N\right|^{1+1/p}}{\omega_2\left(f,\textrm{diam}\,T_i^N\right)}
		\geqslant 
		\displaystyle \frac{c_2^p\varepsilon^{p/2}}{\omega_2^p\left(f,\varepsilon N^{-1/4}\right)} \sum\limits_{i\in M_3(\triangle_N;\varepsilon)}\left|T_i^N\right|^{p+1}. 
	\end{array}
$$

Using convexity of the function $t^{p+1}$ and assumption~(\ref{1assump}), we have 
$$
	\mathcal{F}(\triangle_N) 
	> 
	\displaystyle\frac{c_2^p\varepsilon^{p/2}}{\omega_2^p\left(f,\varepsilon N^{-1/4}\right)} \sum\limits_{i\in M_3(\triangle_N;\varepsilon)}\left|T_i^N\right|^{p+1} 
	\geqslant 
	\displaystyle\frac{c_2^p\varepsilon^{1+3p/2}\left(\# M_3(\triangle_N;\varepsilon)\right)^{-p}}{\omega_2^p\left(f,\varepsilon N^{-1/4}\right)} 
	\geqslant 
	\frac{c_2^p\varepsilon^{1 + 3p/2}}{N^p \omega_2^p\left(f,\varepsilon N^{-1/4}\right)}. 
$$
Therefore, 
$$
	\liminf\limits_{N\to\infty} N\cdot\inf\limits_{s\in\mathcal{S}(\triangle_N)} \|f-s\|_{p;\alpha,\beta} 
	\geqslant 
	c_2\varepsilon^{3/2 + 1/p}\cdot \lim\limits_{N\to\infty} \omega_2^{-1}\left(f,\varepsilon N^{-1/4}\right) 
	= 
	+\infty. 
$$
The latter contradicts the assumption~(\ref{assumption_of_lemma_X3}). The lemma is proved. 
\end{proof}

\begin{lemma}
\label{L12}
Let $\varepsilon>0$ and let $\left\{\triangle_N\right\}_{N=1}^{\infty}$ be the sequence of triangulations $\triangle_N = \left\{T_i^N\right\}_{i=1}^{N}$ of $D$. If~(\ref{assumption_of_lemma_X3}) holds true then for all $N$ large enough inequality~(\ref{lower_area_bound}) also holds true. 
\end{lemma}

\begin{proof} 
Let number $\varepsilon>0$ be fixed. Assume to the contrary that there exists a subsequence $\{N_k\}_{k=1}^{\infty}$ of positive integers, such that $N_k\to \infty$ as $k\to \infty$ and inequality~(\ref{1assump}) holds true. Without loss of generality we let $N_k=k$, $k\in\NN$. By Lemma~\ref{L10} and the definition of the set $M_4(\triangle_N;\varepsilon)$ for every $N\in\NN$, we have
$$
	\begin{array}{rcl}
		\mathcal{F}(\triangle_N) 
			& 
		:= 
			& 
		\displaystyle\inf\limits_{s\in\mathcal{S}(\triangle_N)} \|f-s\|_{p;\alpha,\beta}^p
		\geqslant 
		\sum\limits_{i=1}^N E^p(f,\mathcal{P}_1)_{L_{p;\alpha,\beta}\left(T_i^N\right)} 
		\geqslant 
		\displaystyle \sum\limits_{i\in M_4(\triangle_N;\varepsilon)} \varepsilon^p \Upsilon_f^{p} \left(\textrm{diam}\,T_i^N\right)^{2p} \left|T_i^N\right| 
		\\
			& 
		\geqslant 
			& 
		\displaystyle \varepsilon^{3p}\Upsilon^p_f N^{-p/2} \sum\limits_{i\in M_4(\triangle_N;\varepsilon)} \left|T_i^N\right| 
		\geqslant 
		\varepsilon^{3p+1}\Upsilon^p_f N^{-p/2}. 
	\end{array}
$$
Therefore, 
$$
	\liminf_{N\to\infty} N\cdot\inf\limits_{s\in\mathcal{S}(\triangle_N)} \|f-s\|_{p;\alpha,\beta} 
	\geqslant 
	\varepsilon^{3+1/p}\Upsilon_f \lim\limits_{N\to\infty} \sqrt{N} 
	= 
	+\infty.
$$
The latter contradicts to assumption~(\ref{assumption_of_lemma_X3}) of lemma. The lemma is proved. 
\end{proof}

\begin{lemma}
\label{L13}
Let $\varepsilon>0$ and let $\left\{\triangle_N\right\}_{N=1}^{\infty}$ be the sequence of triangulations $\triangle_N = \left\{T_i^N\right\}_{i=1}^{N}$ of $D$. If ~(\ref{assumption_of_lemma_X3}) holds true then for all $N$ large enough inequality~(\ref{lower_area_bound}) also holds true. 
\end{lemma}

\begin{proof} 
Let number $\varepsilon>0$ be fixed and let $\delta_0$ be the minimal positive number such that $\omega(H,\delta_0) = \varepsilon/6$. Assume to the contrary that there exists a subsequence $\{N_k\}_{k=1}^{\infty}$ of positive integers, such that $N_k\to \infty$ as $k\to \infty$ and inequality~(\ref{1assump}) holds true. Without loss of generality we let $N_k=k$ for every $k\in\NN$. 

By the definition of the set $M_5(\triangle_N;\varepsilon)$, for every $i\in M_5(\triangle_N;\varepsilon)$ there exist points $L_1,L_2\in T_i^N$ such that $H(f;L_1)<\varepsilon$ and $H(f;L_2)\geqslant 2\varepsilon$. By $L$ we denote the point on the segment $L_1 L_2$ such that $H(f;L)=3\varepsilon/2$. We also denote by $B$ the ball centered at the point $L$ with the radii $\delta_0$. Evidently, 
$$
	D\cap B \subset A_{2\varepsilon}\cap F_{\varepsilon}. 
$$
Hence, $H(f;x,y)\geqslant \varepsilon$ for every point $(x,y)\in D\cap B$. By Lemma~\ref{L10} there exists a triangle $\widetilde{T}_i^N\subset B\cap T_i^N$ such that $\left|\widetilde{T}_i^N\right| \geqslant K_{\delta_0}^2 \left|T_i^N\right|$ and $\textrm{diam}\,\widetilde{T}_i^N \geqslant K_{\delta_0}\textrm{diam}\,T_i^N$ where $K_{\delta_0}$ was defined in Lemma~\ref{L10}. In addition, $\textrm{diam}\,T_i^N\geqslant 2\delta_0$. Then in view of Lemma~\ref{L}, 
$$
	\begin{array}{rcl}
		\mathcal{F}(\triangle_N) 
			& 
		:= 
			& 
		\displaystyle\inf\limits_{s\in\mathcal{S}(\triangle_N)} \|f-s\|_{p;\alpha,\beta}^p
		\geqslant 
		\sum\limits_{i=1}^N E^p(f,\mathcal{P}_1)_{L_{p;\alpha,\beta}\left(T_i^N\right)} 
		\geqslant 
		\displaystyle \sum\limits_{i\in M_5(\triangle_N;\varepsilon)} E^p(f,\mathcal{P}_1)_{L_{p;\alpha,\beta}\left(T_i^N\right)} 
		\\
			& 
		\geqslant 
			& 
		\displaystyle\sum\limits_{i\in M_5(\triangle_N;\varepsilon)}  E^p(f,\mathcal{P}_1)_{L_{p;\alpha,\beta}\left(\widetilde{T}_i^N\right)} 
		\geqslant 
		\displaystyle \sum\limits_{i\in M_5(\triangle_N;\varepsilon)} \varepsilon^p\Upsilon_f^p\left(\textrm{diam}\,\widetilde{T}_i^N\right)^{2p}\left|\widetilde{T}_i^N\right| 
		\\ 
			& 
		\geqslant 
			& 
		\displaystyle \varepsilon^p\Upsilon_f^p K_{\delta_0}^{2p+2} \sum\limits_{i\in M_5(\triangle_N;\varepsilon)} \left(\textrm{diam}\, T_i^N\right)^{2p}\left|T_i^N\right| 
		\geqslant \varepsilon^{p+1}\Upsilon_f^p K_{\delta_0}^{2p+2} (2\delta_0)^{2p}. 
	\end{array}
$$
Therefore, 
$$
	\liminf\limits_{N\to\infty} N\cdot\inf\limits_{s\in \mathcal{S}(\triangle_N)} \|f-s\|_{p;\alpha,\beta} 
	= 
	+\infty 
$$
which contradicts to assumption~(\ref{assumption_of_lemma_X3}) of lemma. The lemma is proved. 
\end{proof} 

Now we have all facts needed to prove the lower estimate~(\ref{lower_estimate}). To that end we need to show that for every sequence $\{\triangle_N\}_{N=1}^{\infty}$ of triangulations $\triangle_N=\{T_i^N\}_{i=1}^N$ of $D$,
$$
	\liminf\limits_{N\to\infty} N\cdot\inf\limits_{s\in\mathcal{S}(\triangle_N)} \|f-s\|_{p;\alpha,\beta} 
	\geqslant 
	2^{-1} C_{p;\alpha,\beta}\left\|\sqrt{H}\right\|_{\frac{p}{p+1}}. 
$$
Without loss of generality we consider only those sequences $\{\triangle_N\}_{N=1}^{\infty}$ for which
$$
	\liminf\limits_{N\to\infty} N\cdot\inf\limits_{s\in\mathcal{S}(\triangle_N)} \|f-s\|_{p;\alpha,\beta} 
	< 
	\infty. 
$$

For every $i\in M_2(\triangle_N;\varepsilon)$ we substitute the function $f$ on the triangle $T_i^N$ by its Taylor polynomial of second order $f_{N,i}$ constructed at the point $U_{T_i^{N}}$. In view of Lemma~\ref{L3}, we have 
$$
	\begin{array}{rcl}
		\left\|f-f_{N,i}\right\|_{L_{p; \alpha , \beta }(T^{N}_i)} 
			& 
		\leqslant 
			& 
		\max\{\alpha;\beta\}\|f-f_{N,i}\|_{L_{\infty}(T_i^{N})}|T_i^{N}|^{1/p} 
		\\ [5pt] 
			& 
		\leqslant 
			& 
		2\max\{\alpha;\beta\}(\textrm{diam}\,T_i^N)^2\omega_2(f, \textrm{diam}\,T_i^N)|T_i^{N}|^{1/p}.
	\end{array}
$$
By the definition of the set $M_2(\triangle_N;\varepsilon)$, 
$$
	\left\|f-f_{N,i}\right\|_{L_{p; \alpha , \beta }\left(T^{N}_i\right)} 
	\leqslant 
	2^{-1}\varepsilon  C_{p;\alpha,\beta}\cdot H^{1/2}\left(f;U_{T_i^N}\right)\left|T_i^N\right|^{1+1/p}.
$$
Applying the triangle inequality and Corollary~\ref{cor}, we obtain 
$$
	\begin{array}{rcl}
		\displaystyle E(f,\mathcal{P}_1)_{L_{p;\alpha,\beta}\left(T_i^N\right)} 
			& 
		\geqslant 
			& 
		\displaystyle E(f_{N;i},\mathcal{P}_1)_{L_{p;\alpha,\beta}\left(T_i^N\right)} -\|f-f_{N;i}\|_{L_{p;\alpha,\beta}\left(T_i^N\right)} 
		\\ 
			& 
		\geqslant 
			& 
		\displaystyle 2^{-1}(1-\varepsilon)C_{p;\alpha,\beta}\cdot H^{1/2}\left(f;U_{T_i^N}\right)\left|T_i^N\right|^{1+1/p}. 
	\end{array}
$$
Now, 
$$
	\begin{array}{rcl}
		\mathcal{F}(\triangle_N) 
			& 
		:= 
			& 
		\displaystyle\inf\limits_{s\in\mathcal{S}(\triangle_N)} \|f-s\|_{p;\alpha,\beta}^p
		\geqslant 
		\sum\limits_{i=1}^N E^p(f,\mathcal{P}_1)_{L_{p;\alpha,\beta}\left(T_i^N\right)} 
		\geqslant 
		\displaystyle \sum\limits_{i\in M_2(\triangle_N;\varepsilon)} E^p(f,\mathcal{P}_1)_{L_{p;\alpha,\beta}\left(T_i^N\right)} 
		\\
			& 
		\geqslant 
			& 
		\displaystyle 2^{-p}(1-\varepsilon)^p C^p_{p;\alpha,\beta} \sum\limits_{i\in M_2(\triangle_N;\varepsilon)}H^{ p/2}\left(f;U_{T_i^N}\right)\left|T_i^N\right|^{p+1}. 
	\end{array}
$$
Application of the Jensen inequality for the function $t^{p+1}$ implies that 
$$
	\begin{array}{rcl}
		\mathcal{F}(\triangle_N) 
			& 
		\geqslant 
			& 
		\displaystyle \frac{(1-\varepsilon)^pC^p_{p;\alpha,\beta}}{2^p\left(\# M_2(\triangle_N;\varepsilon)\right)^p}  \left(\sum\limits_{i\in M_2(\triangle_N;\varepsilon)} H^{\frac{p}{2(p+1)}}\left(f;U_{T_i^N}\right)\left|T_i^N\right|\right)^{p+1}
		\\ 
			& 
		\geqslant 
			& 
		\displaystyle \frac{(1-\varepsilon)^pC^p_{p;\alpha,\beta}}{2^pN^p}  \left(\sum\limits_{i\in M_2(\triangle_N;\varepsilon)} H^{\frac{p}{2(p+1)}}\left(f;U_{T_i^N}\right)\left|T_i^N\right|\right)^{p+1}. 
	\end{array}
$$

Let us subdivide each triangle $T_i^{N}$, $i\in I_N\setminus M_2(\triangle_N;\varepsilon)$, into $n_i^{N}$ smaller triangles $T_{i,j}^{N}$, $j=1,\ldots,n_i^{N}$, enumerated in arbitrary order, such that $\textrm{diam}\,T_{i,j}^{N}\to 0$ as $N\to\infty$ for all $j=1,\ldots,n_i^{N}$. For every $i\in M_2(\triangle_N;\varepsilon)$, set $n_i^{N}:=1$ and $T_{i,1}^{N}:=T_i^{N}$. Observe that $\bigcup\limits_{i=1}^{N}\bigcup\limits_{j=1}^{n_i^{N}}T_{i,j}^{N}=D$, and for every appropriate $i$ and $j$, it follows that $\textrm{diam}\,T_{i,j}^{N}\to 0$ as $N\to\infty$. Then 
$$
	\sum\limits_{i\in M_1(\triangle_N;\varepsilon)}\sum_{j=1}^{n_i^N} H^{\frac{p}{2(p+1)}}\left(f;U_{T_{i,j}^N}\right)\left|T_{i,j}^N\right| 
	\leqslant 
	(2\varepsilon)^p \sum\limits_{i\in M_1(\triangle_N;\varepsilon)} \left|T_i^N\right| 
	\leqslant 
	(2\varepsilon)^{p}\cdot 1 
	= 
	(2\varepsilon)^p. 
$$
In addition, by Lemmas~\ref{L11},~\ref{L12}, and~\ref{L13} for all $N$ large enough we have for $r=3,4,5$,
$$
	\sum\limits_{i\in M_r(\triangle_N;\varepsilon)}\sum_{j=1}^{n_i^N} H^{\frac{p}{2(p+1)}}\left(f;U_{T_{i,j}^N}\right)\left|T_{i,j}^N\right| 
	\leqslant 
	\|H\|_{\infty}^{\frac p{2(p+1)}} \sum\limits_{i\in M_r(\triangle_N;\varepsilon)} \left|T_i^N\right| 
	\leqslant 
	\varepsilon\|H\|_{\infty}^{\frac p{2(p+1)}}.
$$

Therefore, 
$$
	\begin{array}{rcl}
		\mathcal{F}(\triangle_N) 
			& 
		\geqslant 
			&
		\displaystyle\left(\frac{(1-\varepsilon)C_{p;\alpha,\beta}}{2N}\right)^p\left(\sum_{i=1}^{N}\sum_{j=1}^{n_i^{N}} H^{\frac{p}{2(p+1)}}\left(f;U_{T_{i,j}^{N}}\right)\left|T_{i,j}^{N}\right|  -  3\varepsilon\|H\|_{\infty}^{\frac p{2(p+1)}} - (2\varepsilon)^{p}\right)^{p+1}
		\\
			& 
		= 
			& 
		\displaystyle\left(\frac{C_{p;\alpha,\beta}}{2N}\right)^p\left(\int_D H^{\frac{p}{2(p+1)}}(f;x,y)\,dx\,dy + o(1) - 3\varepsilon\|H\|_{\infty}^{\frac p{2(p+1)}} - (2\varepsilon)^{p}\right)^{p+1},
	\end{array}
$$
as $N\to\infty$. Hence,
$$
	\liminf_{N\to\infty}N^p \cdot\inf\limits_{s\in S(\triangle_N)}\|f-s\|^p_{p; \alpha , \beta}
	\geqslant 
	2^{-p}C^p_{p;\alpha,\beta}\left(\int_D H^{\frac{p}{2(p+1)}}(f;x,y) \,dx\,dy - 3\varepsilon\|H\|_{\infty}^{\frac p{2(p+1)}} - (2\varepsilon)^{p}\right)^{p+1} . 
$$
Since $\varepsilon$ is arbitrary, we obtain the desired  inequality~(\ref{lower_estimate}). %\end{proof}

%\section{The proof of Theorem~1 for $p=\infty$} \label{S9}

%In this section we will show that the assertion of Theorem \ref{Th2} is valid in the case $p=\infty$. To this end recall that for every continuous function $g:D\to\RR$, the following inequality holds true
%$$
%  \|g\|_{p;\alpha,\beta} \le \|g\|_{\infty;\alpha,\beta} |D|^{\frac 1p} = \|g\|_{\infty;\alpha,\beta}.
%$$
%Hence, for every number $N\in\NN$, triangulation $\triangle_N$ of the set $D$ and the spline $s\in\mathcal{S}(\triangle_N)$ we have 
%$$
%  \|f-s\|_{p;\alpha,\beta} \le \|f-s\|_{\infty;\alpha,\beta}. 
%$$
%Therefore, for every $p\in[1,\infty)$, we have 
%$$
%  \begin{array}{rcl}
% \displaystyle \liminf\limits_{N\to\infty} N\cdot R_N\left(f,L_{\infty;\alpha,\beta}\right) & = & \displaystyle\liminf\limits_{N\to\infty} N\cdot \inf\limits_{\triangle_N}\inf\limits_{s\in \mathcal{S}(\triangle_N)}\|f-s\|_{\infty;\alpha,\beta} \ge \liminf\limits_{N\to\infty} N\cdot \inf\limits_{\triangle_N}\inf\limits_{s\in \mathcal{S}(\triangle_N)}\|f-s\|_{p;\alpha,\beta} \\ 
%  & \ge & \displaystyle\liminf\limits_{N\to\infty} N\cdot R_N\left(f,L_{p;\alpha,\beta}\right)
%  \end{array}
%$$
%We conclude from the latter inequality that 
%$$
%  \liminf\limits_{N\to\infty} N\cdot R_N\left(f,L_{\infty;\alpha,\beta}\right) \ge \lim\limits_{p\to\infty} \frac{C_{p;\alpha,\beta}}{2} \left\|\sqrt{H}\right\|_{\frac{p}{p+1}} = \frac{C_{\infty;\alpha,\beta}}{2}\left\|\sqrt{H}\right\|_{\infty;\alpha,\beta}. 
%$$

%The upper estimate in Theorem \ref{Th2} for the case $p=\infty$ we can obtain by applying necessary changes to Section \ref{S7}. %\end{proof}

\section{Asymptotically optimal sequences of triangulations}
\label{S10}

We now choose the asymptotically optimal sequence of triangulations $\{\triangle_N^*\}_{N=1}^{\infty}$ and the corresponding splines $\{s_N^*\}_{N=1}^{\infty}$, that will give the answer to the second question addressed in this paper. 

Let $\{\varepsilon\}_{k=1}^{\infty}$ be a decreasing sequence of positive numbers which tends to zero as $k\to \infty$. In Section \ref{S7.2} for every $\varepsilon>0$, we have constructed sequences of triangulations $\left\{\triangle_N^{\varepsilon}\right\}$ of the set $D$ and corresponding piecewise linear splines $\left\{s_N^{\varepsilon}\right\}$ for every $N$ large enough. By $N(\varepsilon)$ let us denote the minimal number $N$ for which the triangulation $\Delta_N^{\varepsilon}$ and $s_N^{\varepsilon}$ were constructed. Without loss of generality we may assume that the sequence of numbers $\{N\left(\varepsilon_k\right)\}_{k=1}^{\infty}$ is strictly increasing. Then, we set 
$$
	\triangle_N^*
	:=
	\triangle_N^{\varepsilon_k},\quad 
	s_N^*
	:=
	s_N^{\varepsilon_k},\qquad \textrm{if}\qquad N(\varepsilon_k)<N\leqslant N(\varepsilon_{k+1}), \quad k\in \NN,
$$
where triangulations $\triangle_N^{\varepsilon}$ and splines $s_N^{\varepsilon}$ were defined in Section~\ref{S7.2}. For $1\leqslant N\leqslant \widetilde{N}(\varepsilon_1)$, we may take $\triangle_N^*$ to be an arbitrary triangulation of $D$, and $s_N^*\in S(\triangle_N^*)$ to be an arbitrary spline.

The above constructed sequence will be asymptotically optimal. Indeed, because of~(\ref{new_ineq}), for all $1\leqslant p<\infty$ and for every $N(\varepsilon_k)<N\leqslant N(\varepsilon_{k+1})$, we have 
$$
	R_N\left(f,L_{p; \alpha , \beta }\right)
	\leqslant 
	\left\|f-s_N^*\right\|_{p; \alpha , \beta } 
	\leqslant 
	\frac{\left(1+k_2\varepsilon_k\right)C_{p;\alpha,\beta}}{2N}\left( \int_{D}H^{\frac{p}{2(p+1)}}(f;x,y)\,dx dy\right )^{1+1/p}.
$$
On the other hand, for every $N$ large enough, 
$$
	R_N(f,L_{p;\alpha,\beta})
	\geqslant 
	\frac{C_{p;\alpha,\beta}}{2\left(1+k_3\varepsilon_k\right)N}\left(\int_D H^{\frac{p}{2(p+1)}}(f;x,y)\,dx\,dy\right)^{1+1/p}. 
$$
Combining the last two inequalities and letting $k\to\infty$, we obtain the desired
$$
	\lim\limits_{N\to\infty} N \cdot R_N(f,L_{p;\alpha,\beta}) 
	= 
	\lim\limits_{N\to\infty} N\|f - s_N^*\|_{p;\alpha,\beta}.
$$


\begin{thebibliography}

\bibitem{Bab82} 
{Babenko V. F. (1982)}\hfill \break
{Non-symmetric approximations in spaces of summable functions}, Ukrain. Mat. Zh., {\bf 34} (1982), pp. 409--416; English transl. Ukrainian Math. J., {\bf 34}, pp. 323--336. 

\bibitem{Bab83} 
{Babenko V. F. (1983)} \hfill \break
{Asymmetric extremal problems in approximation theory}, Dokl. USSR {\bf 269}(3), pp. 521--524. [in Russian]

\bibitem{Bab84} 
{Babenko V. F. (1984)} \hfill \break
{Duality theorems for some problems in approximation theory,} Contemp. questions of real and complex analysis, Kiev, In-t math. AN USSR, pp. 3--13. [in Russian]

\bibitem{Bab87} 
{Babenko V. F. (1987)} \hfill \break
{Approximations, widths and optimal quadrature formulae for classes of periodic functions with rearrangement invariant sets of derivatives}, Anal. Math. {\bf 13}, pp. 15--28.

\bibitem{us} 
{Babenko V., Babenko Y., Ligun A.,  Shumeiko A. (2006)}\hfill \break
{On asymptotical behavior of the optimal linear spline interpolation error of $C^2$ functions}, East J. Approx., {\bf 12}(1), pp. 71--101. 

\bibitem{BBS} 
{Babenko V., Babenko Y., Skorokhodov D. (2008)}\hfill \break
{Exact asymptotics of the optimal $L_{p,\Omega}$-error of linear spline interpolation}, East J. Approx. {\bf 10}(3), pp. 285--237. 

\bibitem{simplex} 
{Babenko V., Babenko Y., Parfinovych N., Skorokhodov D. (2009)}\hfill \break
On one extremal property of a regular simplex, Comm. Anal. Geom. {\bf 17}, no. 4, pp. 685--699.

\bibitem{PhD} 
{Babenko Y. (2006)} \hfill \break
{On the asymptotic behavior of the optimal error of spline interpolation of multivariate functions}, PhD thesis. 

\bibitem{boro} 
{B$\rm{\ddot{o}}$r$\rm{\ddot{o}}$czky K., Ludwig M. (1999)} \hfill \break
{Approximation of Convex Bodies and a Momentum Lemma for Power Diagrams}, Monatshefte f$\rm{\ddot{u}}$r Mathematik, {\bf 127}(2), pp. 101--110. 

\bibitem{Bor} 
{B$\rm{\ddot{o}}$r$\rm{\ddot{o}}$czky K. (2000)} \hfill \break
 {Approximation of general smooth convex bodies},  Adv. in Math., {153} pp. 325--341.

\bibitem{Brezin92}
 {Brezin M. (1992)}\hfill \break
{A solution-based triangular and tetrahedral mesh quality indicator}, SIAM Journal on Scientific Computing {\bf 19}, pp. 979--997. 

\bibitem{chen1} 
{Chen L., Sun P., Xu J. (2007)}\hfill \break
{Optimal anisotropic meshes for minimizing interpolation errors in $L_p$-norm}, Math. Comp. {\bf 76}, pp. 179--204. 

\bibitem{chen} 
{Chen L. (2008)}\hfill \break
{On minimizing the linear interpolation error of convex quadratic functions and the optimal simplex}, East J. Approx. {\bf 10}(3), pp. 271--284. 

\bibitem{Cohen}
{Cohen A.,  Mirebeau J.-M. (2009)} \hfill \break
 {Adaptive and anisotropic piecewise polynomial approximation},
chapter 4 in {\it Multiscale, Nonlinear and Adaptive Approximation}, Springer.



\bibitem{Daz3} 
{D'Azevedo E. F. ,  Simpson R. B. (1989)} \hfill \break
{On optimal interpolation triangle incidences (1989)}, SIAM J. Sci. Statist. Comput. {\bf 10}(6), pp. 1063--1075. 

\bibitem{Dolzhenko1} 
{Dolzhenko E. P., Sevast'yanov E. A. (1998)} \hfill \break
{Approximations with a sign-sensitive weight: existence and uniqueness theorems,} Izv. RAS, Ser. Math., {\bf 62}:6, pp. 59--102; English transl.: Izv. Math., {\bf 62}:2, pp. 1127--1168.   

\bibitem{Dolzhenko2} 
{Dolzhenko E. P., Sevast'yanov E. A. (1999)} \hfill \break
{Approximations with a sign-sensitive weight. Stability, applications to the theory of snakes and Hausdorff approximations,}
 Izv. RAS, Ser. Math., {\bf 63}:3, pp. 77--118; English transl.: Izv. Math., {\bf 63}:3, pp. 495--534. 

\bibitem{NiraDyn1}
{Dyn N., Levin D., Rippa S. (1990)} \hfill \break
{ Data dependent triangulations
for piecewise linear interpolation,}  IMA J. Numer. Anal., 10,
no. 1, pp. 137--154.

\bibitem{NiraDyn2}
{Dyn N., Levin D., Rippa S. (1992)}\hfill \break
{Boundary correction for piecewise linear interpolation defined over data-dependent triangulations}, Journal of Computational and Applied Mathematics, 39, pp. 179--192.


\bibitem{Toth} {Fejes Toth L. (1972)}\hfill \break
 {Lagerungen in der Ebene, auf der Kugel und im Raum}, 2nd ed. Berlin: Springer.

\bibitem{Handbook}
{Goodman J., O'Rourke J., (eds.) (2004)} \hfill \break
{Handbook of Discrete
and Computational Geometry}, CRC Press.


\bibitem{Gr2} 
{Gruber P. (1988)}\hfill \break
 {Volume approximation of convex bodies by inscribed polytopes}, Math. Ann., {\bf 281}, pp. 229--245.

\bibitem{huang}
{Huang, W.; Sun, W. (2003)}\hfill \break
 {Variational mesh adaptation. II. Error estimates and monitor functions.} J. Comput. Phys. 184, no. 2, pp. 619--648.



\bibitem{Kil}
{Kilizhekov Yu. A. (1996)} \hfill \break
{Approximation error for linear polynomial interpolation on $n$-simplices}, Math. Notes, 60:4, pp. 378--382.

\bibitem{Korn} 
{Korneichuk N. P. (1987)}\hfill \break
{Exact constants in approximation theory}, Nauka, Moscow; translated from Russian by K. Ivanov. Encyclopedia of Mathematics and its Applications, 38. Cambridge University Press, Cambridge, 1991. 

\bibitem{Krein} 
{Krein M. G. (1962)}\hfill \break
{The L-Problem in an abstract linear normed space}, in: Some Questions in the Theory of Moments, N.\,I.~Akhiezer and M.\,G.~Krein (eds.), Am. Math. Soc., Providence, pp. 175--204.

\bibitem{KreinNudelman}
{Krein M.G. ,  Nudel'man A. A. (1977)} \hfill \break
{The Markov Moment Problem and Extremal Problems},
Translations of Mathematical Monographs, V. 50; 417 pp.


\bibitem{JM}
{Mirebeau J.-M.  (2010)} \hfill \break
{Optimally adapted finite elements meshes}, Constructive Approximation, Vol 32, N. 2, pp. 339--383.


\bibitem{Nadler86} 
{Nadler E. (1986)}\hfill \break
 {Piecewise linear best $L_2$ approximation on triangles,}  in: Chui, C.K., Schumaker, L.L. and Ward, J.D. (eds.), Approximation Theory V, Academic Press, pp. 499--502.

\bibitem{kodla1} 
{Pottmann H., Krasauskas R., Hamann B, Joy K., Seibold W. (2000)}\hfill \break
{On piecewise linear approximation of quadratic functions}, J. Geom. Graph. {\bf 4}(1), pp. 23--53.


%\bibitem{Rajan_91} V.\,T.~Rajan, Optimality of Delaunay triangulation in $\RR^d$, Proc. of the Seventh Annual Symp. on Comp. Geom (1991), 357--363. 

\bibitem{Sub}
{Subbotin Yu. N. (1989)} \hfill \break
{The dependence of estimates of a multidimensional piecewise-polynomial approximation on the geometric characteristics of a triangulation}, A work collection of the All-Union school on function theory (Dushanbe, August 1986), Trudy Mat. Inst. Steklov., 189, Nauka, Moscow, pp. 117 -- 137. 

\end{thebibliography}
\end{document}